\documentclass[11pt]{article}
\usepackage{arxiv}
\usepackage{amsfonts}
\usepackage{amsmath}
\usepackage{amsthm}
\usepackage{pxfonts}
\usepackage{mathtools}
\usepackage{tocloft}
\usepackage{todonotes}
\usepackage{extarrows}
\usepackage{titlesec}
\usepackage{enumitem}
\usepackage{subcaption}

\usepackage{hyperref}

\theoremstyle{definition}
\newtheorem{definition}{Definition}[]

\newtheorem{remark}[definition]{Remark}

\newtheorem{assumption}[definition]{Assumption}

\theoremstyle{plain}
\newtheorem{theorem}[definition]{Theorem}
\newtheorem{corollary}[definition]{Corollary}
\newtheorem{proposition}[definition]{Proposition}
\newtheorem{lemma}[definition]{Lemma}

\begin{document}

\title{Three Dimensional Optimization of Scaffold Porosities for Bone Tissue Engineering}

\author{
    Patrick Dondl \\
    University of Freiburg \\
    \texttt{patrick.dondl@mathematik.uni-freiburg.de}
   \And
    Marius Zeinhofer \\
    University of Freiburg\\
    \texttt{marius.zeinhofer@mathematik.uni-freiburg.de}
}

\maketitle

\begin{abstract}
    We consider the scaffold design optimization problem associated to the three dimensional, time dependent model for scaffold mediated bone regeneration considered in \cite{dondl2021efficient}. We prove existence of optimal scaffold designs and present numerical evidence that optimized scaffolds mitigate stress shielding effects from exterior fixation of the scaffold at the defect site.
\end{abstract}
\keywords{Scaffold Mediated Bone Growth, Optimizing Bone Scaffolds, PDE constrained optimization, Optimal Control}

\section{Introduction}
In this work we extend our previously proposed model from \cite{dondl2021efficient} for bone regeneration in the presence of a bioresorbable porous scaffold to include a scaffold architecture optimization problem that allows geometry and patient dependent optimal scaffold designs. We prove the existence of optimal scaffold density distributions under assumptions on the model's geometry and boundary conditions that are realistic for applications, i.e., non-smooth domains and mixed Dirichlet-Neumann boundary conditions. Furthermore, we present numerical simulations that indicate how to mitigate the negative impact of stress shielding on the bone regeneration process that appears in vivo due to the external fixation of the scaffold at the defect site, as documented in, e.g., \cite{sumner1992determinants, huiskes1992relationship, behrens2008numerical, arabnejad2017fully} in the context of total hip arthroplasty or \cite{viateau07, Terjesen09} for femoral defects.

Our model was previously proposed in \cite{poh2019optimization} and analysed in \cite{dondl2021efficient}. Its essential processes are an interplay between the mechanical and biological environment which we model by a coupled system of PDEs and ODEs. The mechanical environment is represented by a linear elastic equation and the biological environment through reaction-diffusion equations as well as logistic ODEs, modelling signalling molecules and cells/bone respectively. Material properties are incorporated using homogenized quantities not resolving any scaffold microstructure. This makes the model efficient in computations and thus allows to solve the scaffold architecture optimization problem numerically with manageable computational cost.

The article is organized as follows. In follwoing section, we provide a brief introduction to tissue engineering for the treatment of severe bone defects, present our computational model and the corresponding scaffold optimization problem. Then, we discuss its weak formulation in Section~\ref{section:mathematical_formulation}, followed by the presentation of the main analytical results in Section~\ref{section:main_results} together with an outline of their proofs. We proceed with the discussion of the numerical simulations in Section~\ref{section:simulations}. The detailed proofs of the main analytical results are provided in Appendix~\ref{section:proofs_of_the_main_results}.

\subsection{Scaffold Mediated Bone Growth}\label{section:scaffold_mediated_bone_growth}
The treatment of critical-sized bone defects ($>$25 mm) and restoration of skeletal functions is challenging using current treatment options, see \cite{nauth2018critical}. Non-unions of the defect, i.e.,\ when no bridging is achieved after $>$9 months and healing stagnates for 3 months, pose severe problems for patients, as discussed in \cite{calori2017non}, and have a relevant financial dimension. For the UK \cite{stewart2019fracture} estimates the healthcare costs to $\pounds 320$ million annually. Further, healing prospects can be worsened by comorbidities such as diabetes, which is associated with compromised bone regeneration capacity, see \cite{marin2018impact}.

Recent research illustrates the potential of porous, bio-resorbable support structures, as temporary support. These structures are called scaffolds an are implanted in the defect site to provide stability, allow vascularization and guidance for new bone formation. Promising results were recently obtained in vivo and in clinical studies. We refer to \cite{petersen2018biomaterial, cipitria2012porous, paris2017scaffold, petersen2018biomaterial, pobloth2018mechanobiologically}. These works also indicate that material choice and scaffold design are critical variables for a successful healing outcome yet they are at present not fully understood.

Several objectives need to be considered for the scaffold design, including (a) the pore size, porosity and shape of the microstructure, which has a significant influence on vascularization, cell proliferation and cell differentiation; (b) the mechanical properties of the scaffold need to guarantee a proper strain distribution within the defect site in order for bone to grow; (c) information concerning comorbitities of the patients that imply reduced bone growths or bone density as caused by diabetes need to be incorporated in the scaffold design. Hence, designing scaffolds that are optimized for the patient and defect site at hand are of fundamental importance. With the possibilities of additive manufacturing, personalized scaffold designs are within reach.

So far, the question of optimal scaffold design has be dominated by trial-and-error approaches. However, this workflow is expensive and prohibits patient specific designs. Topology optimization techniques have successfully been exploited for optimal design questions yielding scaffolds that meet elastic optimality with given porosity or fluid permeability, as shown in \cite{dias2014optimization, coelho2015bioresorbable, lin2004novel, guest2006optimizing, challis2012computationally, kang2010topology,Wang:2016js, dondl_bone_shape_opt}. However, these methods usually lack the ability to resolve the time dependence of the elastic moduli in scaffold mediated bone growth that are integral to incorporate for an optimal scaffold design.

Based on the previous studies \cite{poh2019optimization, dondl2021efficient}, we propose a PDE constrained optimization problem for the optimal porosity (or, equivalently, density) distribution of a scaffold. The advantages of optimized scaffolds include (a), the mitigation of stress shielding effects, which are caused by external fixation of the defect site and result in areas of low stress within the scaffold. This leads to poor bone regeneration, as mechanical stimulus is indispensable for bone growth. (b) Optimized scaffolds can be designed patient dependent by altering the model's parameters. It is straight forward to include a reduced bone regeneration capability due to, e.g., diabetes into the model.

Note that the model proposed in \cite{dondl2021efficient} does not resolve the porous micro-structure of the scaffold design, but uses coarse-grained values instead, i.e., values averaged over a volume representative of the scaffold microstructure (a representative volume element, RVE), and thus optimizes only these macroscopic quantities. In a scaffold based on a unit cell design, the scaffold volume fraction (or equivalently, the porosity) changes on a larger length-scale than the unit cell design. Likewise, the other quantities of the model can be viewed as locally averaged values. We stress that this homogenization viewpoint of the model is essential for the feasibility of the optimization problem and we refer to \cite{dondl2021efficient} for more information concerning this viewpoint. The resulting optimal scaffold porosity distributions can easily be introduced in a 3d printable, periodic microstructure based, scaffold design: one simply considers a one-parameter family of microstructures, where the thickness of struts or surfaces in the microstructure is given by the parameter. The parameter can then be chosen non-constant over the scaffold domain such that at each periodic unit cell the correct scaffold density is recovered as an average.

The main contributions of this article are as follows. We analyze the optimal scaffold design problem mathematically and prove the existence of optimal scaffold designs under realistic assumptions on the geometry of the computational domain, the boundary conditions and functional relationships in Section~\ref{section:main_results}. The existence of such scaffolds is a necessary prerequisite for a successful numerical treatment of the optimization problem. In Section~\ref{section:simulations} we then investigate numerically how optimized scaffold designs mitigate the negative impacts of stress shielding. Our findings suggest that our three dimensional optimization routine is able to successfully mitigate adverse stress shielding effects that may arise from external scaffold fixation.

\subsection{The System of Equations}\label{section:the_system_of_equations}
The underlying paradigm of the model is that an interplay of the biological and the mechanical environment are responsible for bone growth where the mechanical environment is described through displacements and strains and the biological environment through bio-active molecules (signalling molecules) and different cell types. The model is a coupled system of evolution equations composed of a linear elastic equilibrium equation for every point in time, diffusion equations for the bio-active molecules and ordinary differential equations for the concentration of osteoblasts and the volume fraction of bone. The model considered here is a concrete instance of the more general model proposed in \cite{dondl2021efficient}. 

By $\Omega \subset \mathbb{R}^3$ we denote the computational domain, i.e.,\ the bone defect site. We let $I = [0,T]$ be a finite time interval. On $\Omega$ we keep track of the local scaffold volume fraction which we call $\rho(x)$, where $x\in\Omega$. Therefore the porosity $\theta$ is $\theta(x) = 1 - \rho(x)$, but we use only $\rho$. We do not resolve a time dependency for $\rho$. This is due to the experimental findings in \cite{pitt1981aliphatic} which have shown that, in the time-frame of two years, PCL degrades largely due to bulk erosion. Only the molecular mass decreases which we model by the exponential decay $\sigma (t) = e^{-k_1 t}$. Thus product $\rho(x)\cdot\sigma(t)$ quantifies the mechanical properties of PCL over time and space. We denote the bone volume fraction averaged over a RVE by $b(t,x)$ and the variables $b, \sigma$ and $\rho$ then determine the material properties of the bone-scaffold composite. We work in the linear elastic regime and use the notation $\mathbb{C}(\rho,\sigma,b)$ for the elastic tensor capturing the material properties.

Besides isotropy and ellipticity, we assume little for the tensor $\mathbb{C}(\rho,\sigma,b)$. Choosing a microstructure allows to explicitely specify $\mathbb{C}(\rho, \sigma, b)$ in numerical applications. By $u(t,x)$ we denote the displacement field satisfying mechanical equilibrium equations, see \eqref{StrongElasticEquation}. We denote the strain by $\varepsilon(u)$. 
	
We represent the biological environment through two bio-active molecules $a_1(t,x)$ and $a_2(t,x)$, which should be viewed as endogenous angiogenic and osteoinductive factors. These molecules diffuse depending on the scaffold density $\rho$ which we model by $D(\rho)$ in the equation \eqref{StrongDiffusionEquation}. We leave this as an abstract functional relationship for the same reasons as discussed in the context of the elastic tensor. Further, we assume exponential decay and production of bio-active molecules in the presence of strain and a local density of osteoblast cells which are denoted by $c(t,x)$. Essential for the production of growth factors is mechanical stimulus $S(\varepsilon(u))$. The quantity $S(\varepsilon(u))$ is derived from strain, admissible choices include the strains magnitude or octahedral shear strains. We allow flexibility in the choice of $S$, see also the discussion in Section \ref{section:mathematical_formulation}. The strain dependent reaction term is motivated by Wolff's law for bone remodeling, see \cite{wolff1892gesetz}. Strain as a driving force for bone regeneration is also supported by more recent work, for example in \cite{ruff2006s}. Note that the concentration of bio-active molecules is normalized to unity in healthy tissue. This dictates the decay and production rates in concrete simulations.

The equation \eqref{StrongCellODE} governs the production of osteoinductive cells (here: osteoblasts) and is modeled by logistic growth with a driving factor depending on both bio-active molecules $a_1$ and $a_2$ and a proliferation term $(1+k_7c)$. The factor $1-c(1-\rho)^{-1}$ encodes the osteoblast carrying capacity and implies that the osteoblast density is bounded by the amount of available space, i.e., the space not filled by the scaffold. We neglect diffusion in this equation as osteoblasts diffuse on a significantly lower level than the bio-active molecules. Modeling only one cell type in our present model is a simplification and an extension of the model is easily feasible. The equation for bone growth \eqref{StrongODE} is similar to the one for osteoblast concentration. We remark that osteoblast and bone do not compete for space, this encodes the assumption that, e.g., $b=1$ and $c=1$ means that osteoblasts reside ``in saturation'' within healthy bone. Our system of equations is 
\begin{align}
		0 &= \operatorname{div}\big(
		\mathbb{C}(\rho,\sigma,b)\varepsilon(u)
		\big) \label{StrongElasticEquation}
	    &\parbox{14em}{(mechanical equilibrium)}
		\\
		d_ta_i 
		&=
		\label{StrongDiffusionEquation}
		\operatorname{div}\big(
		D(\rho)\nabla a_i
		\big)
		+
		k_{2,i}S(\varepsilon(u)) c
		-
		k_{3,i}a_i
		&\parbox{14em}{(diffusion, generation, and decay of $i=1,2$ bio-molecules)}
		\\
		d_tc 
		&=
		\label{StrongCellODE}
		k_6a_1a_2(1+k_7c)\bigg(1 - \frac{c}{1-\rho}
		\bigg)
		&\parbox{14em}{(osteoblast generation)}
		\\
		d_tb
		&=		
		\label{StrongODE}
		k_4a_1c\bigg(
		1-\frac{b}{1-\rho}
		\bigg)
		&\parbox{14em}{(bone regeneration driven by $a,b$ and $c$).}
	\end{align}
	In the above system $k_1,k_{2,i}, k_{3,i}, k_4, k_6, k_7\geq0$, $i=1,2$ are constants that need to be determined from experiments, compare to Section \ref{section:simulations} where we discuss certain choices. The functional relationships $\mathbb{C}, D(\rho)$ and $S(\cdot)$ are all required to satisfy certain technical assumptions that guarantee the well-posedness of the above system. We discuss this in detail in Section \ref{section:mathematical_formulation}. For concrete examples of the functional relationships we refer to \cite{dondl2021efficient}.

Finally, we need to specify boundary conditions. For the elastic equilibrium equation we allow mixed boundary conditions including the limiting cases of a pure displacement boundary condition and a pure stress boundary condition. As for the bio-active molecules we assume that these are in saturation, i.e., $a(t,x) = 1$ adjacent to bone and on the rest of the boundary of $\Omega$ we assume no-flux boundary conditions. For the initial time-point we propose $a_i(0,x) = a_{i,0} = 0$, for $i=1,2$ inside of $\Omega$. This choice reflects the scenario of a scaffold that is not preseeded with exogenous growth factors. However, different choices of $a_{i,0}$ are admissible and allow the model to cover e.g., pre-seeding with osteoinductive factors. Finally, at the initial time we assume that no osteoblasts and no regenerated bone are present inside the domain of computation. In formulas, it holds for all $i = 1,2$
	\begin{align}
		a_i(0,x)\label{InitialMolecules}
		&=
		0
		&\parbox{18em}{for all $x\in\Omega$}
		\\
		a_i(t,x)\label{BoundaryDirichletMolecules}
		&=
		1
		&\parbox{18em}{for all $t\in I$, $x$ adjacent to healthy bone}
		\\
		D_i^\rho\nabla a_i(t,x)\cdot\eta
		&=
		0
		&\parbox{18em}{for all $t\in I$, $x$ not adjacent to healthy bone} \label{BoundaryNeumannMolecules}
		\\ 
		\big(
		\mathbb{C}(\rho,\sigma,b)\varepsilon(u(t,x))
		\big)\cdot\eta
		&=
		g_N(x)
		&\parbox{18em}{on the Neumann boundary of $\Omega$}
		\\ 
		u(t,x)
		&=
		g_D(x)
		&\parbox{18em}{on the Dirichlet boundary of $\Omega$}
		\\ 
		c(0,x) = b(0,x)\label{InitialBone}
		&=
		0
		&\parbox{18em}{for all $x\in \Omega$.}
	\end{align}
	One may also consider Robin type boundary conditions for the diffusion equations instead of equation \eqref{BoundaryNeumannMolecules}. The model allows for a time dependent choice of the mechanical loading $g_D$ and $g_N$. Due to the long regeneration time horizon of approximately $12$ months, however, it is not expedient to resolve very short time-scales of, e.g., the mechanics of physical therapy. Instead, we consider suitably time-averaged loading conditions here.
\subsection{The Optimization Problem}\label{section:the_optimization_problem}
    In the system \eqref{StrongElasticEquation} - \eqref{StrongODE} above, the function $\rho$, i.e., the scaffold's volume fraction, is a design parameter -- also called control variable -- that we can control in applications. For example, a given scaffold volume fraction distribution could be additively manufactured. Given a certain control variable $\rho$, we denote the solution of the system \eqref{StrongElasticEquation} - \eqref{StrongODE} by  
    \begin{equation*}
     y_\rho \coloneqq (u_\rho,a^1_\rho, a_\rho^2,c_\rho,b_\rho)   
    \end{equation*}
    to stress the dependency on $\rho$. We will also use the notation $\phi(\rho) = y_\rho$ and refer to $\phi$ as the solution operator of the system \eqref{StrongElasticEquation} - \eqref{StrongODE}. Depending on the state $y_\rho$, we can measure the control variable's performance by the value of an objective function $J$ evaluated at $\rho$ and $y_\rho$. We are interested in minimizing or maximizing the objective function over the set of admissible control variables. In other words, we are interested in the optimization problem of finding
    \begin{equation}\label{eq:optimization_problem_heuristically}
         \underset{\rho}{\operatorname{argmin}}J(\rho,y_\rho) \quad \text{subjected to} \quad \rho\in P,
    \end{equation}
	where the set $P$ encodes for example that $\rho$ takes values in the unit interval (necessary for a reasonable volume fraction). The fact that corresponding to $\rho$, we consider the solution $y_\rho$ makes this a PDE-constrained optimization problem and $\rho \in P$ introduces box constraints on the control variable. The concrete form of $J$ is an engineering choice. For instance, the amount of regenerated bone at a certain time-point in the healing process should be maximized. Another alternative we pursue is to maximize the temporal minimum of the elastic modulus. In the case of a hard load for the elastic equation, the elastic modulus at a time-point $t$ is proportional to the elastic energy $\mathcal{E}(t)$, i.e.,
	\begin{equation*}
	    \mathcal{E}_\rho(t) = \frac12\int_\Omega \mathbb{C}(\rho,\sigma(t),b(t))\varepsilon(u(t)):\varepsilon(u(t))\mathrm dx,
	\end{equation*}
	where $u$ and $b$ solve the system \eqref{StrongElasticEquation} - \eqref{StrongODE} corresponding to $\rho$. The minimum of $\mathcal{E}$ over the whole regeneration process describes the weakest state of the bone-scaffold structure during healing. This gives rise to the objective function
	\begin{equation*}
	    \hat J (\rho) = \min_{t\in I}\mathcal{E}_\rho(t)
	\end{equation*}
	and the \emph{maximization} problem of finding
	\begin{equation*}
	   \rho^* \in \underset{\rho \in P}{\operatorname{argmax}}\hat J(\rho).
	\end{equation*}
	The set $P$ encodes pointwise constraints on $\rho$, i.e., the necessity of enforcing $\rho(x) \in [0,1]$ in order to be a meaningful volume fraction. The notation $\hat J$ instead of $J$ is chosen to indicate that the variables $u$ \& $b$ appearing in the definition of $\mathcal{E}_\rho$ are solving the system \eqref{StrongElasticEquation} - \eqref{StrongODE}. Usually, $\hat J$ is called the reduced objective function to distinguish it from the objective function $J$ that does not require $u$ and $b$ to solve the PDE system. If we use a soft load instead of a hard load, the elastic modulus is proportional to the inverse of the elastic energy, hence the objective function becomes
	\begin{equation*}
	    \hat J(\rho) = \max_{t\in I}\mathcal{E}_\rho(t)
	\end{equation*}
	and the optimization consists of finding
	\begin{equation*}
	    \rho^* \in \underset{\rho \in P}{\operatorname{argmin}}\hat J(\rho), 
	\end{equation*}
	i.e., is a \emph{minimization} problem.
	\begin{remark}
	    The proposed objective functions are not smooth as they involve minimizing or maximizing over $t\in I$. For a numerical implementation, one might therefore approximate the minimum or maximum functional by an $L^p(I)$ norm with large value for $-p$ or $p$ respectively.
	\end{remark}
	Another choice of objective function is to consider the amount of regenerated bone after a given time $T$. This results in the definition
	\begin{equation*}
	    \hat J(\rho) = \int_\Omega b(T)\mathrm dx
	\end{equation*}
	\begin{remark}
	    Care needs to be taken with respect to the functional relationships in the system \eqref{StrongElasticEquation} - \eqref{StrongODE} when choosing the amount of regenerated bone as an objective. This requires an adequate choice of $S(\cdot)$. If $S(\cdot)$ is chosen to be the Frobenius norm, the above objective function promotes very weak scaffolds as these lead to high strains and high bone growth. A more sensible choice for $S(\cdot)$ in this case is to use a filter, i.e., only strains with a certain range of magnitude lead to non-vanishing values of $S(\cdot)$. 
	\end{remark}

\section{Mathematical Formulation}\label{section:mathematical_formulation}

\subsection{Notation and Preliminaries}\label{section:notation_and_preliminaries}
By $X$ we usually denote a generic Banach space and with $X^*$ we denote its dual space. The dual pairing for $f\in X^*$ and $x\in X$ is denoted by $\langle f,x \rangle_X$. If a sequence $(x_n)_{n\in\mathbb{N}}$ converges weakly in $X$ to $x$ we write $x_n\rightharpoonup x$. For an interval $I=[0,T]$ and $p\in[1,\infty]$ we denote the Bochner space of $p$-integrable functions with values in $X$ by $L^p(I,X)$, see for instance \cite{diestel1977vector}. Further, we write $W^{1,p}(I,X,Y)$ for the vector valued Sobolev space consisting of functions $u\in L^p(I,X)$ with distributional derivative $d_tu \in L^p(I,Y)$ where $X$ and $Y$ are Banach spaces with $X\hookrightarrow Y$, see for instance \cite{boyer2012mathematical}.

We say a bounded, open set $\Omega\subset\mathbb{R}^d$ is a Lipschitz domain if $\overline{\Omega}$ is a Lipschitz manifold with boundary, see \cite[Definition 1.2.1.2]{grisvard2011elliptic}.
In the following we will denote the cube $[-1,1]^n\subset\mathbb{R}^d$ by $Q$, its half $\{x\in Q\mid x_d <0\}$ by $Q_-$, the hyperplane $\{ x\in Q \mid x_d = 0 \}$ by $\Sigma$ and  $\{x\in\Sigma\mid x_{d-1}<0\}$ by $\Sigma_0$. The following definition is due to Gr\"oger, see \cite{groger1989aw}.
\begin{definition}[Gr\"oger Regular Sets]\label{definition:groger_regular_sets}
    Let $\Omega\subset\mathbb{R}^d$ be bounded and open and $\Gamma\subset\partial\Omega$ a relatively open set. We call $\Omega\cup\Gamma$ Gr\"oger regular, if for every $x\in\partial\Omega$ there are open sets $U,V\subset\mathbb{R}^d$ with $x\in U$, and a bijective, bi-Lipschitz map $\phi:U\to V$, such that $\phi(x) = 0$ and $\phi(U\cap (\Omega\cup\Gamma) )$ is either $Q_-$, $Q_-\cup\Sigma$ or $Q_-\cup\Sigma_0$.
\end{definition}
It can easily be seen that a Gr\"oger regular set $\Omega$ (no matter the choice $\Gamma \subset \partial\Omega$) is a Lipschitz domain, see \cite[Theorem 5.1]{haller2009holder}. The requirement of Gr\"oger regularity is very mild and all applications we have in mind fall in this category. This claim is justified by the following useful characterization of Gr\"oger regular sets in two and three dimensions that allow to check Gr\"oger regularity almost ``by appearance''. The results are due to \cite{haller2009holder}.
\begin{theorem}[[Gr\"oger Regular Sets in 2D, Theorem 5.2 in \cite{haller2009holder}]]\label{theorem:groger_regular_sets_in_2d}
    Let $\Omega \subset \mathbb{R}^2$ be a Lipschitz domain and $\Gamma\subset\partial\Omega$ be relatively open. Then $\Omega\cup\Gamma$ is Gr\"oger regular if and only if $\overline{\Gamma}\cap(\partial\Omega\setminus\Gamma)$ is finite and no connected component of $\partial\Omega\setminus\Gamma$ consists of a single point.
\end{theorem}
\begin{theorem}[Gr\"oger Regular Sets in 3D, Theorem 5.4 in \cite{haller2009holder}]\label{theorem:groger_regular_sets_in_3d}
    Let $\Omega \subset \mathbb{R}^3$ be a Lipschitz domain and $\Gamma\subset\partial\Omega$ be relatively open. Then $\Omega\cup\Gamma$ is Gr\"oger regular if and only if the following two conditions hold
    \begin{itemize}
        \item [(i)] $\partial\Omega\setminus\Gamma$ is the closure of its interior.
        \item [(ii)] For any $x\in \overline{\Gamma}\cap(\partial\Omega\setminus\Gamma)$ there is an open neighborhood $U_x$ of $x$ and a bi-Lipschitz map $\phi:U_x\cap\overline{\Gamma}\cap(\partial\Omega\setminus\Gamma)\to (-1,1)$.
    \end{itemize}
\end{theorem}

\subsection{Setting}\label{section:setting_optimal_control}
In this Section we state the precise framework we use for the optimal control result. We begin by specifying the assumptions on the domain.

\textbf{The Domain.} We consider a finite time interval $I = [0,T]$. The spatial domain $\Omega \subset \mathbb{R}^d$ with $d = 1,2,3$ is assumed to be an open, bounded and connected Lipschitz domain. We consider partitions of the boundary $\partial\Omega$, namely
\begin{gather*}
    \partial\Omega = \Gamma_N^e \cup \partial_D^e, \quad\text{and}\quad \partial\Omega = \Gamma_N^d\cup\Gamma_D^d
\end{gather*}
that will be used for the elastic and the diffusion equation respectively. For the partition of the elastic equation we assume $|\Gamma_D^e| \neq 0$. We require both $\Omega \cup \Gamma_N^e$ and $\Omega \cup \Gamma_N^d$ to be Gr\"oger regular, see Definition~\ref{definition:groger_regular_sets} or \cite{groger1989aw} and \cite{haller2009holder}.

\begin{remark}
    Note the following things.
    \begin{itemize}
        \item [(i)] For the elastic equation we exclude a pure Neumann problem, however, we can include this case by passing to a suitable quotient space. We excluded this for convenience and brevity only.
        \item[(ii)] The assumption of Gr\"oger regularity is very mild and all desirable application settings we have in mind easily satisfy this requirement. Compare to \cite{haller2009holder} for more information.
    \end{itemize}
\end{remark}

\textbf{The Control Space.} The set of control variables is defined to be
\begin{equation}\label{equation:set_of_control_variables}
    P = \left\{ \rho \in H^2(\Omega) \mid 0 < c_P \leq \rho(x) \leq C_P < 1 \right\},
\end{equation}
where $c_P$ and $C_P$ are two fixed constants. Note that in the spatial dimensions $d=1,2,3$ the space $H^2(\Omega)$ embeds into $C^0(\Omega)$, hence the pointwise condition imposed in the above definition is well-defined. 

\textbf{The State Space and the Equations.} Consider the state space
\begin{equation*}
    Y = C^0(I,H^1_{D_e}(\Omega)) \times H^1(I,H^1_D(\Omega),H^1_D(\Omega)^*)\cap L^4(I,C^0(\Omega))\times W^{1,2}_0(I,C^0(\Omega))^2
\end{equation*}
and the space
\begin{equation*}
    W = L^2(I,H^1_{D_e}(\Omega))^* \times L^2(I,H^1_{D_d}(\Omega))^* \times L^2(\Omega) \times L^2(I,C^0(\Omega))^2.
\end{equation*}
Then the state equations can be written in the form $e(y,\rho) = 0$ with the constraint operator
\begin{equation*}
    e:Y\times P \to W, \quad (y,\rho) = (\tilde u,\tilde a_1,\tilde a_2, b, c, \rho)\mapsto e(y,\rho)
\end{equation*}
given by
\begin{align}\label{equation:state_equations_optimal_control_setting}
    e(y,\rho) 
    = 
    \begin{pmatrix}
    \iint \mathbb{C}(\rho,\sigma, b)\varepsilon(\tilde u+u_D):\varepsilon(\cdot)\mathrm dx\mathrm dt - \int_I \int_{\partial\Omega} g_N\cdot \mathrm ds\mathrm dt
    \\
    \\
    \int_I\langle d_t\tilde a_1,\cdot\rangle_{H^1_{D_d}(\Omega)}\mathrm dt + \iint D(\rho)\nabla \tilde a_1\nabla\cdot + k_{3,1}(\tilde a_1+1)\cdot\mathrm dx\mathrm dt - \iint k_{2,1}S(\varepsilon(u_0+u_D)) c\cdot\mathrm dx\mathrm dt
    \\
    \\
    \int_I\langle d_t\tilde a_2,\cdot\rangle_{H^1_{D_d}(\Omega)}\mathrm dt + \iint D(\rho)\nabla \tilde a_2\nabla\cdot + k_{3,2}(\tilde a_2+1)\cdot\mathrm dx\mathrm dt - \iint k_{2,2}S(\varepsilon(u_0+u_D)) c\cdot\mathrm dx\mathrm dt
    \\
    \\
    \tilde a_1(0) + 1
    \\
    \\
    \tilde a_2(0) + 1
    \\
    \\
    d_tc - k_6(\tilde a_1 + 1)(\tilde a_2 + 1)(1+k_7c)\left( 1 - \frac{c}{1 - \rho} \right)
    \\
    \\
    d_tb - k_4(\tilde a_1 + 1)c\left( 1 - \frac{b}{1 - \rho} \right)
    \end{pmatrix}
\end{align}
We frequently use the notation
\begin{equation*}
    a_i = \tilde a_i + 1, \quad \text{and} \quad u = \tilde u + u_D.
\end{equation*}

\textbf{Functional Relationships.} To make fully sense of the above definition of $e$ we still need to clarify the assumptions made on the data and functional relationships. We begin with the function $\sigma$. We assume that it is smooth, depends only on time and is bounded away from zero, i.e., 
\begin{equation}\label{equation:assumption_on_sigma}
    \sigma\in C^\infty(I), \quad\text{with }\sigma(t) > 0 \text{ for all }t\in I.
\end{equation}
Usually, we set $\sigma$ to be an exponential decay. For the material properties $\mathbb{C}$ of the elastic equation we require that it is a map
\begin{equation*}
    \mathbb{C}:\operatorname{dom}(\mathbb{C})\subset C^0(I\times \Omega)\times C^0(\Omega) \to C^0(I,L^\infty(\Omega,\mathcal{L}(\mathcal{M}_s))), \quad \text{with}\quad (b,\rho)\mapsto \mathbb{C}(b,\sigma,\rho).
\end{equation*}
The concrete definition of $\operatorname{dom}(\mathbb{C})$ is not so important, however, as a minimal requirement it should hold
\begin{equation*}
    \bigcup_{\rho\in P}\{ b\in C^0(I\times\Omega) \mid 0 \leq b(t,x) \leq 1 - \rho(x) \}\times\{\rho\} \subset \operatorname{dom}(\mathbb{C}).
\end{equation*}
Furthermore, we need $\mathbb{C}(\cdot,\sigma,\rho)$ to be Lipschitz continuous with Lipschitz constant independent of $\rho$ and  $\mathbb{C}$ is assumed to be continuous on all of $\operatorname{dom}(\mathbb{C})$. Finally, we require
\begin{equation}\label{equation:assumption_on_C_boundedness}
    \sup_{(b,\rho)\in\operatorname{dom}(\mathbb{C})}\lVert \mathbb{C}(b,\sigma,\rho) \rVert_{L^\infty(\Omega,\mathcal{L}(\mathcal{M}_s))} < \infty
\end{equation}
and
\begin{equation}\label{equation:assumption_on_C_ellipticity}
    \inf_{(b,\rho)\in\operatorname{dom}(\mathbb{C})}\left[\inf_{M\in\mathcal{M}_s\setminus \{0\}}\mathbb{C}(b,\sigma,\rho)M:M\right] \geq c_{\mathbb{C}}|M|^2,
\end{equation}
for a constant $c_{\mathbb{C}}>0$. We need a further regularity property of $\mathbb{C}$. We assume that $b(t)\in C^\alpha(\Omega)$, $\rho\in C^\alpha(\Omega)$ for an $\alpha \in (0,1)$ implies that the coefficient functions
\begin{equation*}
    C_{ijkl}(t)\coloneqq \left[ \mathbb{C}(b,\sigma,\rho)(t) \right]_{ijkl}
\end{equation*}
are members of $C^\alpha(\Omega)$ and that there exists a constant $C>0$ not depending on $b$ and $\rho$ such that
\begin{equation}\label{equation:holder_regulariy_of_coefficients}
    \lVert C_{ijkl}(t) \rVert_{C^\alpha(\Omega)} \leq C\lVert b(t) \rVert_{C^\alpha(\Omega)}\lVert \rho \rVert_{C^\alpha(\Omega)}.
\end{equation}
For the boundary data $u_D$ and $g_N$ of the elliptic equation we assume that 
\begin{equation}\label{equation:assumption_on_g_N}
    g_N \in C^0(I,L^2(\partial\Omega))
\end{equation}
and that the Dirichlet boundary data is given through a function 
\begin{equation}\label{equation:assumption_on_u_D}
    u_D \in C^0(I,H^{1+\theta}(\Omega)),
\end{equation}
meaning that the boundary information can be lifted to all of $\Omega$ such that the lift has the above regularity in time and space, where $\theta > 0$ can be arbitrarily small. In practice, this is easy to verify as we mainly work with Dirichlet boundary conditions that do not vary in time. The material properties $D(\rho)$ used in the diffusion equation are a map
\begin{equation*}
    D:\operatorname{dom}(D)\subset C^0(\Omega) \to L^\infty(\Omega,\mathcal{M}_s), \quad\text{with}\quad \rho\mapsto D(\rho)
\end{equation*}
that we require to be continuous with respect to the uniform norm on $\operatorname{dom}(D)$. The domain of $D$ will usually satisfy
\begin{equation*}
    \{ \rho\in C^0(\Omega) \mid 0 < c_P \leq \rho(x) \leq C_P < 1 \} \subset \operatorname{dom}(D),
\end{equation*}
where $c_P$ and $C_P$ are the positive constants appearing in the definition of $P$. We also require $D$ to be uniformly elliptic independently of $\rho\in\operatorname{dom}(D)$, i.e.,
\begin{equation}\label{equation:assumption_on_D_ellipticity}
    \inf_{\rho\in\operatorname{dom}(D)}\left[ \inf_{\xi \in \mathbb{R}^d\setminus \{0\}} D(\rho)\xi\cdot\xi \right] \geq c_D |\xi|^2,
\end{equation}
for a constant $c_D > 0$. Finally, for the function $S(\cdot)$ we assume that it is given through a map on matrices
\begin{equation*}
    S(\cdot):\mathbb{R}^{d\times d} \to [0,\infty)
\end{equation*}
that we require to be Lipschitz and to obey an estimate of the form
\begin{equation}\label{equation:assumption_on_smoothing_of_norm_for_symmetric_gradient}
    S(A) \leq C_1|A| + C_2\quad\text{for all }A\in\mathbb{R}^{d\times d},
\end{equation}
where $C_1, C_2 >0$ and $|A|$ denotes the Euclidean (or any) norm of a matrix. Furthermore, we need $S$ to be continuous, more precisely, we assume that if $(v_k)\subset L^2(\Omega,\mathbb{R}^{d\times d})$ is a sequence, then it holds
\begin{equation}\label{equation:assumption_on_weak_continuity_of_smoothing_of_norm}
    v_k \to v \quad\text{in }L^2(\Omega,\mathbb{R}^{d\times d}) \quad \Rightarrow \quad S(v_k) \to S(v) \quad\text{in }L^2(\Omega).
\end{equation}

We recall the main result of the first chapter concerning the well-posedness of the PDE-ODE system.
\begin{theorem}
    Assume that the setting described in this section holds. Then, for every $\rho \in P$ there exists a unique solution $y=(\tilde u, \tilde a_1, \tilde a_2, b, c)\in Y$ satisfying $e(y,\rho) = 0$, i.e., solving the state equations \eqref{equation:state_equations_optimal_control_setting}.
\end{theorem}
\begin{proof}
    This follows almost as an application of Theorem 3.2 in \cite{dondl2021efficient}. Note that our assumptions here are slightly stronger, so the requirements in \cite{dondl2021efficient} are trivially satisfied. The only extension to the results in \cite{dondl2021efficient} is to show the improved integrability for the functions $a_1$ and $a_2 \in L^4(I,C^0(\Omega))$. To this end, we use the estimate $(3.11)$ in \cite{dondl2021efficient} which shows that the right-hand sides of the diffusion equations satisfy
    \begin{equation*}
        k_{2,i}S(\varepsilon(u_0+u_D)) c \in L^\infty(I,L^2(\Omega)).
    \end{equation*}
    This allows to apply the maximal $L^p$ regularity result in Lemma~\ref{lemma:the_bound_for_a} and obtain $\tilde a_i \in L^p(I,C^0(\Omega))$ for all $p\in [2,\infty)$.
\end{proof}

\subsection{Objective Function}\label{section:objective_function}
Here we formulate the class of objective functions we are able to treat in the setting of the optimal control result. For every time-point $t\in I$ and state control pair $(y,\rho) \in Y \times P$ we consider the elastic energy
\begin{equation}
    \mathcal{E}:Y\times P \to C^0(I), \quad \mathcal{E}(y,\rho)(t) = t\mapsto \frac12 \int_\Omega \mathbb{C}(b(t),\sigma(t),\rho)\varepsilon(u(t)):\varepsilon(u(t))\mathrm dx.
\end{equation}
For most of our objective functions we desire $\mathcal{E}$ to take values in $C^0(I)$, as we want to have access to point evaluations. This is the reason to require the continuity of the solutions to the elastic equation in the definition of $Y$. Primarily, we are interested in the reduced elastic energy $\hat{\mathcal{E}}$, that is, we are interested in $\mathcal{E}(y,\rho)$ only when $(y,\rho)$ solves the system of equations, i.e., when it holds $e(y,\rho) = 0$. We define
\begin{equation}
    \hat{\mathcal{E}}: P \to C^0(I), \quad \hat{\mathcal{E}}(\rho) = \mathcal{E}(y_\rho,\rho) 
\end{equation}
and here it holds $e(y_\rho,\rho) = 0$. We provide now the proof that $\mathcal{E}$ takes values in $C^0(I)$.
\begin{lemma}\label{lemma:strict_positivity_of_elastic_energy}   
    For all $(y,\rho)\in Y\times P$ we have $\mathcal{E}(y,\rho)\in C^0(I)$. If it holds $e(y,\rho) = 0$ and $\tilde u(t) +u_D(t) \neq 0$, then $\mathcal{E}(y,\rho)(t) > 0$. 
\end{lemma}
\begin{proof}
    As $\tilde u+u_D \in C^0(I, H^1(\Omega))$ by the definition of the state space $Y$ and the material tensor $\mathbb{C}$ is a member of the space $C^0(I,L^\infty(\Omega,\mathcal{L}(\mathcal{M}_s)))$ it follows that 
    \begin{equation*}
        \mathbb{C}(\sigma, \rho, b)\varepsilon(\tilde u+u_D):\varepsilon(\tilde u+u_D) \in C^0(I,L^1(\Omega)).
    \end{equation*}
    Using the $L^1(I)$ continuity of integration, we get $\mathcal{E}(y,\rho)\in C^0(I)$. Now, let $e(y,\rho) = 0$. We can estimate
    \begin{equation*}
        \mathcal{E}(y,\rho)(t) \geq c \lVert \tilde u(t) + u_D \rVert^2_{H^1(D)},
    \end{equation*}
    with the constant $c > 0$ depending on the constant appearing in Korn's inequality and the ellipticity constant $c_{\mathbb{C}}$. As it holds $e(y,\rho) = 0$, for every $t\in I$ the function $\tilde u(t) + u_D(t)$ solves an elastic equation, hence can only vanish if the boundary conditions are homogeneous for this time-point which leads to $\tilde u(t) + u_D(t) = 0$. This is excluded in the statement of the Lemma and the proof is complete.
\end{proof}
We state now the structural assumption we impose for our admissible objective functions.
\begin{assumption}\label{assumption:objective_function_optimal_control}
    Let $\mathcal{F}:\operatorname{dom}(\mathcal{F})\subset C^0(I) \to \mathbb{R}$ be a continuous map and assume that the domain of $\mathcal{F}$ satisfies
    \begin{equation}\label{equation:domain_of_F}
        \left\{ v \in C^0(I) \mid v(t) > 0 \text{ for all }t\in I \right\} \subset \operatorname{dom}(\mathcal{F}).
    \end{equation}
    Furthermore, let $\mathcal{G}:C^0(I\times \Omega) \to \mathbb{R}$ be a continuous function. Using the elastic energy $\mathcal{E}$ and functionals $\mathcal{F}$, $\mathcal{G}$ as above, we define the prototypical objective function as
    \begin{equation*}
        J: Y \times P \to \mathbb{R}, \quad J(y,\rho) = \mathcal{F}\left( \mathcal{E}(y,\rho) \right) + \mathcal{G}(b)
    \end{equation*}
    in case the domain of $\mathcal{F}$ allows $\mathcal{E}(y,\rho)$ as an argument. The function $b$ denotes the bone component of the state variable $y$. More important, we define the reduced objective
    \begin{equation*}
        \hat J : P \to \mathbb{R}, \quad \hat J(\rho) = \mathcal{F}\left( \mathcal{E}(\phi(\rho),\rho) \right) + \mathcal{G}(b). 
    \end{equation*}
    Note that the assumption \eqref{equation:domain_of_F} together with Lemma \ref{lemma:strict_positivity_of_elastic_energy} guarantees that $\mathcal{E}(\phi(\rho),\rho)$ is an admissible argument of $\mathcal{F}$. Finally, we assume that $\hat J$ is bounded from below if we are interested in a minimization problem and we assume $\hat J$ to be bounded from above if we are interested in maximization.
\end{assumption}
\begin{remark} We discuss how the examples discussed in Section~\ref{section:the_optimization_problem} fall in the abstract setting described above.
    \begin{enumerate}
        \item [(i)] Choosing the minimum (or maximum) functional
        \begin{equation*}
            \min: C^0(I) \to \mathbb{R}, \quad v\mapsto \min_{t\in I}v(t)
        \end{equation*}
        for $\mathcal{F}$ is conforming with Assumption~\ref{assumption:objective_function_optimal_control} as clearly $\operatorname{min}$ and $\operatorname{max}$ are continuous functionals on $C^0(I)$.
        \item[(ii)] Smooth approximations of the minimum and the maximum are given by $L^p(I)$ norms with large values of $|p|$. A positive value for $p$ serves as an approximation of the maximum and a negative value is suitable for the approximation of the minimum. In the latter case, i.e., $p<0$, one chooses 
        \begin{equation*}
            \operatorname{dom}\left( \lVert\cdot\rVert_{L^p(I)} \right) \coloneqq \left\{ v \in C^0(I) \mid v(t) > 0 \text{ for all }t\in I \right\}.
        \end{equation*}
        It is straight forward to show that $\lVert \cdot \rVert_{L^p(I)}$ is continuous with respect to the uniform norm, also for negative exponents. In fact, it is even Fr\'echet differentiable.
        \item[(iii)] The choice
        \begin{equation*}
            \mathcal{G}(b) = \int_\Omega b(T) \mathrm dx,
        \end{equation*}
        corresponds to the objective of regenerated bone at time $T$. Clearly, $\mathcal{G}$ is continuous and evaluating $\mathcal{G}$ only at functions $b$ that solve the state equations shows that $\mathcal{G}$ is bounded.
    \end{enumerate}
\end{remark}

\section{Main Results}\label{section:main_results}
Our main result establishes the existence of an optimal control in the set $P\subset H^2(\Omega)$ given the objective function $\hat J$ is regularized by an $H^2(\Omega)$ norm.
\begin{theorem}[Optimal Control]\label{theorem:optimal_control_approx_objective}
    Assume we are in Setting \ref{section:setting_optimal_control} and let $\eta > 0$ be fixed. Then there exists a minimizer $\rho^* = \rho^*(\eta) \in P$ to the regularized objective 
    \begin{equation*}
        \hat{J}(\rho^*) + \eta\lVert \rho^* \rVert_{H^2(\Omega)} = \inf_{\rho \in P}\left[ \hat{J}(\rho) + \eta\lVert  \rho \rVert^2_{H^2(\Omega)} \right].
    \end{equation*}
\end{theorem}
\begin{proof}
    The proof is established in the course of the article.
\end{proof}

In order to incorporate the pointwise constraint encoded in the definition of the control space $P$, see \eqref{equation:set_of_control_variables}, in a numerical simulation one can use a soft penalization. This usually corresponds to a continuous functional $\mathcal{K}:C^0(\Omega)\to[0,\infty)$. Also in this setting we can establish the existence of an optimal control.
\begin{corollary}\label{corollary:optimal_control_numerical_penalization}
    Assume we are in Setting \ref{section:setting_optimal_control} and let $\mathcal{K}:C^0(\Omega)\to[0,\infty)$ be a continuous, non-negative functional. Then there exists an optimal control $\rho^\dag = \rho^\dag(\eta,\mathcal{K}) \in H^2(\Omega)$ to the regularized and penalized objective, i.e., 
    \begin{equation*}
        \hat J(\rho^\dag) + \eta \lVert \rho^\dag \rVert^2_{H^2(\Omega)} + \mathcal{K}(\rho^\dag) = \inf_{\rho\in H^2(\Omega)}\left[ \hat J(\rho) + \eta \lVert \rho \rVert^2_{H^2(\Omega)} + \mathcal{K}(\rho) \right].
    \end{equation*}
\end{corollary}
\begin{proof}
    The proof is established in the course of the article.
\end{proof}
\begin{remark}
    A few comments regarding the above results are in order.
    \begin{enumerate}
        \item[(i)] For some objectives we might be interested in a maximizer rather than a minimizer. In this case, one subtracts the regularizer $\eta \lVert \cdot \rVert_{H^2(\Omega)}$ and the soft penalty $\mathcal{K}$ and the results are still valid. For brevity, we discuss only minimization problems in the remainder.
        \item [(ii)] As discussed in Section \ref{section:objective_function}, we have some freedom in the choice of $\hat J$. From a modelling perspective a maximum or minimum over all time-points of the elastic energy seems reasonable. On the other hand, for the numerical treatment a smooth approximation thereof is preferable, e.g., an $L^p(I)$ norm. Note that all these choices are covered by our main result.
        \item[(iii)] The Tikhonov penalization term $\eta\lVert\cdot\rVert^2_{H^2(\Omega)}$ is artificial. It serves to generate compactness of minimizing sequences and an optimal control result without this term seems out of reach.
        \item[(iv)] It is presently unclear to us if the optimal control problem possesses a unique solution.
    \end{enumerate}
\end{remark}

The strategy to prove Theorem~\ref{theorem:optimal_control_approx_objective} and Corollary~\ref{corollary:optimal_control_numerical_penalization} is the direct method of the calculus of variations and crucially relies on rather specific regularity properties of the diffusion equations and the elastic equation that imply convenient compact embeddings. The technical results concerning these regularity properties are established in Appendix \ref{section:proofs_of_the_main_results}. In this Section, we assume the implications of the compact embeddings and show how this leads to a proof of Theorem~\ref{theorem:optimal_control_approx_objective}. We stress that the mixed boundary conditions, rough coefficients and jump initial conditions are responsible for the technical difficulties.

\begin{proposition}\label{proposition:proof_under_assumptions} Assume we are in Setting \ref{section:setting_optimal_control}. Let $(\rho_k)\subset P$ be a minimizing sequence for $\hat J + \eta\lVert\cdot\rVert^2_{H^2(\Omega)}$ and denote by $(u_k) \subset C^0(I,H^1(\Omega)) $, $(a^1_k),(a^2_k) \subset H^1(I,H^1(\Omega),H^1_D(\Omega)^*)$ and $(b_k), (c_k) \subset W^{1,2}(I,C^0(\Omega)) $ the corresponding solutions to the system \ref{equation:state_equations_optimal_control_setting}. Assume that there is a common subsequence (not relabeled) of  $(\rho_k), (u_k), (a^1_k), (a^2_k), (b_k), (c_k)$ and elements $\rho^* \in P$, $u^* \in C^0(I,H^1(\Omega))$, $a_1^*$, $a_2^* \in H^1(I,H^1(\Omega),H^1_D(\Omega)^*)$ and $b^*, c^* \in W^{1,2}(I, C^0(\Omega))$ such that 
\begin{itemize}
    \item [(A1)] $\rho_k \to \rho^*$ in $C^0(\Omega)$ and $\rho_k \rightharpoonup \rho^*$ in $H^2(\Omega)$ \label{item:A1},
    \item[(A2)] $u_k \to u^*$ in $C^0(I, H^1(\Omega))$,
    \item[(A3)] $a^i_k \rightharpoonup a_i^*$ in $H^1(I,H^1(\Omega), H^1_D(\Omega)^*)$, \quad $i=1,2$
    \item[(A4)] $b_k \to b^*$ in $C^0(I\times\Omega)$
    \item[(A5)] $c_k \to c^*$ in $C^0(I\times\Omega)$
\end{itemize}
then $(\rho^*, u^*, a_1^*, a_2^*, b^*)$ solves the system \ref{equation:state_equations_optimal_control_setting} and $\rho^*$ is minimizer of $\hat J + \eta\lVert\cdot\rVert^2_{H^2(\Omega)}$ over the set $P$, i.e., satisfies
\begin{equation*}
    \hat{J}(\rho^*) + \eta\lVert \rho^* \rVert_{H^2(\Omega)} = \inf_{\rho \in P}\left[ \hat{J}(\rho) + \eta\lVert  \rho \rVert^2_{H^2(\Omega)} \right].
\end{equation*}
\end{proposition}
\begin{proof}
    There are two things to show. First, we need to guarantee that the tuple $(\rho^*,u^*,a_1^*, a_2^*,b^*)$ still solves the system of equations \ref{equation:state_equations_optimal_control_setting}. And secondly, we need to prove that $\rho^*$ is in fact a minimizer. We start with the second point, assuming for the moment that $(\rho^*,u^*,a_1^*, a_2^*,b^*)$ solves the correct equations. We show that it holds
    \begin{equation*}
        \hat J(\rho^*) + \eta\lVert \rho^* \rVert_{H^2(\Omega)}^2
        \leq 
        \liminf_{k\to\infty}\left[ \hat J(\rho_k) + \eta\lVert \rho_k \rVert^2_{H^2(\Omega)} \right]
        = 
        \min_{\rho\in P}\left[ \hat{J}(\rho) + \eta\lVert  \rho \rVert^2_{H^2(\Omega)} \right],
    \end{equation*}
    that is, the classical lower semi-continuity property required in the application of the direct method of the calculus of variations. Clearly, the map
    \begin{equation*}
        H^2(\Omega)\to \mathbb{R}, \quad \rho\mapsto \eta\lVert \rho \rVert_{H^2(\Omega)}^2
    \end{equation*}
    is convex and norm continuous, hence weakly lower semi-continuous, that is, it holds
    \begin{equation*}
        \eta\lVert \rho^* \rVert_{H^2(\Omega)}^2 \leq \liminf_{k\to\infty}\eta\lVert \rho_k \rVert_{H^2(\Omega)}^2
    \end{equation*}
    by the assumption $\rho_k\rightharpoonup\rho^*$ in $H^2(\Omega)$ on the minimizing sequence. To proceed, remember our structural assumption on the objective function, i.e.,
    \begin{equation*}
        \hat J = \mathcal{F}\left( \hat{\mathcal{E}}(\rho) \right) + \mathcal{G}(b),
    \end{equation*}
    where $\mathcal{F}:C^0(I)\to \mathbb{R}$ and $\mathcal{G}:C^0(I\times\Omega)\to \mathbb{R}$ are assumed to be continuous. Thus it suffices to show that $\mathcal{E}(\rho_k)\to \mathcal{E}(\rho^*)$ in $C^0(I)$. For convenience, let us now set $\mathbb{C}^* = \mathbb{C}(\rho^*,\sigma,b^*)$ and $\mathbb{C}_k = \mathbb{C}(\rho_k,\sigma,b_k)$. We then compute
    \begin{align*}
        \lVert \hat{\mathcal{E}}(\rho_k) - \hat{\mathcal{E}}(\rho^*) \rVert_{C^0(I)}
        &=
        \frac12 \left\lVert \int_\Omega \left[\mathbb{C}_k-\mathbb{C}^*\right]\varepsilon(u_k):\varepsilon(u_k) 
        +
        \mathbb{C}^*\varepsilon(u_k-u^*):\varepsilon(u_k) 
        +
        \mathbb{C}^*\varepsilon(u^*)\varepsilon(u_k-u^*)\mathrm dx
        \right\rVert_{C^0(I)}
        \\
        &\leq
        \lVert \mathbb{C}_k - \mathbb{C}^* \rVert_{C^0(I,L^\infty(\Omega,\mathcal{L}(\mathcal{M}_s)))} \lVert \varepsilon(u_k) \rVert^2_{C^0(I,L^2(\Omega))}
        \\&+
        \lVert \mathbb{C}^* \rVert_{C^0(I,L^\infty(\Omega,\mathcal{L}(\mathcal{M}_s)))} \lVert \varepsilon(u_k - u^*) \rVert^2_{C^0(I,L^2(\Omega))} \lVert \varepsilon(u_k) \rVert^2_{C^0(I,L^2(\Omega))}
        \\&+
        \lVert \mathbb{C}^* \rVert_{C^0(I,L^\infty(\Omega,\mathcal{L}(\mathcal{M}_s)))} \lVert \varepsilon(u^*) \rVert^2_{C^0(I,L^2(\Omega))} \lVert \varepsilon(u_k - u^*) \rVert^2_{C^0(I,L^2(\Omega))}.
    \end{align*}
    Using the continuity assumption for $\mathbb{C}$ and the convergence $b_k\to b^*$ in $C^0(I\times\Omega)$ and $\rho_k\to\rho^*$ in $C^0(\Omega)$ we get that
    \begin{equation*}
        \lVert \mathbb{C}_k - \mathbb{C}^* \rVert_{C^0(I,L^\infty(\Omega,\mathcal{L}(\mathcal{M}_s)))} \to 0.
    \end{equation*}
    Furthermore, the convergence $u_k\to u^*$ in $C^0(I,H^1(\Omega))$ implies both a bound on $\lVert \varepsilon(u_k) \rVert$ and the convergence
    \begin{equation*}
        \lVert \varepsilon(u_k - u^*) \rVert_{C^0(I,L^2(\Omega))}.
    \end{equation*}
    Hence, we established $\hat{\mathcal E}(\rho_k)\to\hat{\mathcal{E}}(\rho^*)$ and conclude
    \begin{equation*}
        \hat J(\rho^*) + \eta\lVert \rho^* \rVert^2_{H^2(\Omega)} \leq \lim_{k\to\infty}\hat J(\rho_k) + \liminf_{k\to\infty}\eta\lVert \rho_k \rVert^2_{H^2(\Omega)} \leq \liminf_{k\to\infty}\left[ \hat J(\rho_k) + \eta\lVert \rho_k \rVert_{H^2(\Omega)}^2 \right]
    \end{equation*}
    which settles the claim. 
    
    We still need to show that $(\rho^*,u^*,a_1^*, a_2^*,b^*)$ is in fact a solution to the system \ref{equation:state_equations_optimal_control_setting}. For the elastic equation we consider for an arbitrary test function $\varphi\in L^2(I,H^1_{D_e}(\Omega))$
    \begin{equation*}
        \iint \mathbb{C}(\rho_k,\sigma,b_k)\varepsilon(u_k):\varepsilon(\varphi)\mathrm dx\mathrm dt = \int_I \int_{\partial\Omega}g_N\varphi\mathrm ds\mathrm dt
    \end{equation*}
    and the continuity assumption on $\mathbb{C}$ and the convergence assumed for $\rho_k$, $b_k$ and $u_k$ are by far sufficient to pass to the limit. 
    
    In the same spirit, we consider the diffusion equations with a test function $\varphi \in L^2(I,H_{D_d}(\Omega))$
    \begin{equation*}
        \int_I\langle d_ta^i_k,\varphi\rangle_{H^1_{D_d}(\Omega)}\mathrm dt + \iint D(\rho_k)\nabla a^i_k\nabla \varphi + k_3 (a_k^i)\varphi \mathrm dx\mathrm dt = \iint k_2 S(\varepsilon(u_k)) c_k\varphi\mathrm dx\mathrm dt, \quad i=1,2.
    \end{equation*}
    For the left-hand side of the diffusion equations we can easily pass to the limit by the weak convergence of $a^i_k$ and the strong convergence of $D(\rho_k)$ that we have available through the continuity assumption on $D$ and $\rho_k\to\rho^*$ in $C^0(\Omega)$. For the right-hand sides we use the implication
    \begin{equation*}
        u_k\to u^*\text{ in }C^0(I,H^1(\Omega))\quad\Rightarrow\quad S(\varepsilon(u_k)) \to S(\varepsilon(u^*)) \text{ in }L^2(\Omega).
    \end{equation*}
    Hence, the limit for the diffusion equations can also be correctly identified. To establish the initial condition of the limit, consider the continuous linear map
    \begin{equation*}
        H^1(I,H^1_{D_d}(\Omega), H^1_{D_d}(\Omega)^*) \to C^0(I,L^2(\Omega)) \to L^2(\Omega), \quad a\mapsto a(0).
    \end{equation*}
    Using the weak sequential continuity of continuous linear maps shows that $a^*(0)$ vanishes, as desired.
    
    To pass to the limit in the cell ODE, we look at its fixed-point equation
    \begin{equation*}
        c_k(t) = \int_0^t k_6a^k_1(s)a^k_2(s)(1 + k_7c_k(s))\left(1 - \frac{c_k(s)}{1 - \rho_k}\right)\mathrm ds,
    \end{equation*}
    which holds in the space $C^0(\Omega)$, for all $t\in I$. Multiplying the above equation by a smooth test function $\varphi \in C_c^\infty(\Omega)$ and integrating over $\Omega$ yields for the left-hand side of the above equation
    \begin{equation*}
        \int_\Omega c_k(t)\varphi\mathrm dx \to \int_\Omega c^*(t)\varphi\mathrm dx \quad \text{with}\quad k\to \infty.
    \end{equation*}
    The convergence $c_k \to c^*$ in the space $C^0(I\times\Omega)$ suffices by far for the above limit passage. Before we treat the limit of the right-hand side we note that the compactness result of Aubin-Lions, see for instance \cite{simon1986compact}, provides the compact embedding
    \begin{equation*}
        H^1(I,H^1_D(\Omega), H^1_D(\Omega)^*) \hookrightarrow\hookrightarrow L^2(I,L^2(\Omega))
    \end{equation*}
    which is essentially due to the fact that the space triple $(H^1_D(\Omega),L^2(\Omega), H^1_D(\Omega)^*)$ satisfies the requirements of the Ehrling Lemma, being in turn guaranteed by the Rellich-Kochandrov compactness result that provides the compact embedding of $H^1_D(\Omega)$ into $L^2(\Omega)$. Note that the boundary regularity in for $\Omega$ is chosen to support the Rellich-Kochandrov theorem. Hence we get the convergence
    \begin{equation*}
        a^k_1 a^k_2 \to a^*_1 a^*_2 \quad\text{in}\quad L^1(I,L^1(\Omega))\Tilde{=}L^1(I\times\Omega).
    \end{equation*}
    Using the above convergence and the convergence of $c_k\to c^*$ in $C^0(I\times\Omega)$ and $\rho_k \to \rho^*$ in $C^0(\Omega)$ we compute, employing Fubini's theorem and pass to the limit
    \begin{align*}
        \int_\Omega \int_0^t k_6a^k_1 a^k_2(1 + k_7c_k)\left(1 - \frac{c_k}{1 - \rho_k}\right)\mathrm ds \varphi \mathrm dx &= \int_0^t\int_\Omega k_6 a^k_1 a^k_2 (1 + k_7c_k)\left(1 - \frac{c_k}{1 - \rho_k}\right)\varphi \mathrm ds \mathrm dx
        \\
        &\to
        \int_0^t\int_\Omega k_6a^*_1 a^*_2 (1 + k_7c^*)\left(1 - \frac{c^*}{1 - \rho^*}\right)\varphi \mathrm dx \mathrm ds
        \\
        &=
        \int_\Omega\int_0^t k_6 a^*_1 a^*_2 (1 + k_7c^*)\left(1 - \frac{c^*}{1 - \rho^*}\right) \mathrm ds \varphi \mathrm dx
    \end{align*}
    Inferring the fundamental lemma of the calculus of variations we obtain
    \begin{equation*}
        c^*(t) = \int_0^t k_6 a_1^* a_2^* (1+k_7c^*)\left( 1 - \frac{c^*}{1-\rho^*} \right)\mathrm ds
    \end{equation*}
    for every $t \in I$. This implies that $c^*$ satisfies the correct limit equation. Obviously we can repeat the same argument to guarantee that $b^*$ satisfies an appropriate limit equation.
\end{proof}
\begin{remark}
    Via discussing the requirements $(A1) - (A5)$ above, we give a rough idea of their proof. 
    \begin{itemize}
        \item [(i)] The fact that $J$ is bounded from below implies that the regularization term $\eta\lVert\cdot\rVert^2_{H^2(\Omega)}$ automatically leads to an $H^2(\Omega)$ bound on any minimizing sequence $(\rho_k)\subset P$. Thus there exists $\rho^* \in P$ and a (not re-labeled) subsequence $(\rho_k)$ with $\rho_k \rightharpoonup \rho^*$ in $H^2(\Omega)$. Employing the compactness
        \begin{equation*}
            H^2(\Omega) \hookrightarrow\hookrightarrow C^0(\Omega)
        \end{equation*}
        that holds for three spatial dimensions, this implies the desired convergence $\rho_k \to \rho^*$ in $C^0(\Omega)$.
        
        \item[(ii)] A uniform bound in $C^0(I,H^1(\Omega))$ norm of the sequence $(u_k)$ is easily established as Lemma \ref{lemma:bound_for_u_k} shows. However, this does not provide assumption (A2) which can only be achieved through a compactness argument. In fact -- given H\"older continuous coefficients functions of $\mathbb{C}(\rho_k,\sigma, b_k)$ -- one is able to show that for every $t\in I$ the solution $u_k(t)$ is a member of $H^{1+\theta}(\Omega)$ for a sufficiently small $\theta >0$ as an application of the main theorem of \cite{haller2019higher}. Compare also to Lemma \ref{lemma:higher_regularity_elliptic_equation} for a discussion of the applicability of this result. Then, given the relative compactness of the sequences $(b_k)$ in $C^0(I\times\Omega)$ and $(\rho_k) \subset C^0(\Omega)$ one can apply a vector-valued version of the Arzel\`a-Ascoli theorem to derive the relative compactness of $(u_k)$ in $C^0(I,H^1(\Omega))$. As discussed in (iv), the compactness of $(b_k)$ relies on a H\"older regularity result for diffusion equations.
        \item[(iii)] Similarly, a uniform bound for the sequences $(a^i_k)$ in $H^1(I,H^1(\Omega),H^1_D(\Omega)^*)$ norm can be established by standard computations, thus implying the desired existence of $a_i^*$ and corresponding subsequence. We provide the details in Lemma \ref{lemma:bound_for_a_k}.
        \item[(iv)] The existence of a subsequence $(b_k)$ and $b^*\in C^0(I\times\Omega)$ with $b_k\to b^*$ in $C^0(I\times\Omega)$ requires the biggest effort. We achieve this by deriving a $W^{1,2}(I,C^\alpha(\Omega))$ bound on $(b_k)$ for an $\alpha \in (0,1)$. Investigating the structure of the cell and bone ODEs, we see that such a regularity and bound can only be established if we are able to show that the sequences $(a^i_k)$ are bounded in $L^2(I,C^\alpha(\Omega))$. It is this regularity and boundedness result for the diffusion equation on which the whole proof rests, we state it in Lemma \ref{lemma:the_bound_for_a}, but the derivation of this result is the topic of \cite{dondl2021regularity}.
        
        Coming back to the boundedness of $(b_k)$ in $W^{1,2}(I,C^\alpha(\Omega))$, note that this implies the desired existence of $b^*\in C^0(I\times\Omega)$ together with a subsequence $b_k \to b^*$ in $C^0(I\times\Omega)$ via the embeddings
        \begin{equation*}
            W^{1,2}(I,C^\alpha(\Omega)) \hookrightarrow C^\beta(I,C^\alpha(\Omega)) \hookrightarrow C^{\min(\alpha,\beta)}(I\times\Omega)\hookrightarrow\hookrightarrow C^0(I\times\Omega).
        \end{equation*}
        \item[(vi)] To summarize: $(A1)$ is clear, $(A3)$ is established in lemma \ref{lemma:bound_for_a_k}, $(A2)$, $(A4)$ and $(A5)$ rely on the regularity result for diffusion equations stated in Lemma \ref{lemma:the_bound_for_a} and the main result of \cite{haller2019higher}. The derivation of the $W^{1,2}(I,C^\alpha(\Omega))$ bound for $(b_k)$ is carried out in Lemma \ref{lemma:bound_for_b_k}, the bound for $(c_k)$ in Lemma \ref{lemma:bound_for_c_k}.
    \end{itemize}
\end{remark}

\section{Simulations}\label{section:simulations}
In this section we present numerical simulations of optimal scaffold density distributions. Our motivation are large tibial defects and we are especially interested in stress shielding effects caused by external fixation of the scaffold. Our numerical findings indicate that a three dimensional scaffold density optimization is of substantial importance in the mitigation of stress shielding effects.

\subsection{Stress Shielding}\label{section:stress_shielding}
Bone adapts according to the mechanical environment it is subjected to. This important property of bone is well known and commonly referred to as Wolff's law, see \cite{wolff1892gesetz}. It has far ranging consequences for bone tissue engineering. More precisely, prosthetic implants are often made of less elastic materials than bone and thus change the mechanical environment in their vicinity. This often leads to bone regions that are subjected to less stress and consequently bone resorption when compared to a healthy bone, a phenomenon known as stress shielding which has been extensively studied, e.g., in the context of total hip arthroplasty, see \cite{sumner1992determinants, huiskes1992relationship, behrens2008numerical, arabnejad2017fully}. The bone resorption in the vicinity of the prosthetic implant can lead to serious complications such as periprosthetic fracture and aseptic loosening and revision surgeries -- if so needed -- can be complicated, we refer to \cite{arabnejad2017fully}.

It is to be expected that stress shielding effects do also play an important role in scaffold mediated bone growth, for example caused through the external fixation of the scaffold by a metal plate. This leads to under-loading in the vicinity of the fixating element. To be able to quantify these effects it is crucial to use a three dimensional computational model, a one dimensional simplification as for instance discussed by \cite{poh2019optimization} cannot resolve the asymmetries that induce the effect. 

\subsection{The Computational Model}
Our concrete model setup is almost identical to the one presented in \cite{dondl2021efficient} as far as the state equations are concerned. For the readers convenience we briefly repeat the state equations and boundary conditions
\begin{align*}
    0 &= \operatorname{div}\Big( \mathbb{C}(\rho,\sigma, b)\varepsilon(u) \Big)
    \\
    d_ta_1 &= \operatorname{div}\Big(  D(\rho) \nabla a_1 \Big) + k_{2,1}|\varepsilon(u)|c - k_{3,1}a_1
    \\
    d_ta_2 &= \operatorname{div}\Big(  D(\rho) \nabla a_2 \Big) + k_{2,2}|\varepsilon(u)|c - k_{3,2}a_2
    \\
    d_tc &= k_6a_1a_2(1 + k_7c)\bigg( 1 - \frac{c}{1 - \rho} \bigg)
    \\
    d_tb &= k_4a_1c\bigg( 1 - \frac{b}{1 - \rho} \bigg).
\end{align*}
We use the same boundary conditions as in \cite{dondl2021efficient} with the exception of the elastic equation that is subjected to pure Neumann boundary conditions with a constant surface traction stemming from a force of $0.3\operatorname{kN}$ which is applied to the top and bottom of the cylindrical domain. We propose to view this as a maximal force that repeatedly occurs, compare to the discussion in \cite{dondl2021efficient} for a more detailed reasoning. The bioactive molecules $a_1,a_2$ are assumed to be in saturation adjacent to the initial, healthy bone matrix at the top and bottom of the domain and a scenario without preseeding throughout the domain (i.e., a zero initial condition) is considered. For the model constants and functional relationships we refer to \cite{dondl2021efficient}.

As an objective function to measure a scaffold performance, we use the maximum over the temporal evolution of the scaffold-bone composite's elastic energy. Due to the softload in the numerical experiments, the reciprocal of the elastic energy is proportional to the elastic modulus of the scaffold-bone system; a reasonable measure of stability. The optimization's goal is to minimize this temporal maximum while respecting the state equations and an additional constraint on $\rho$ to not take values outside the unit interval\footnote{A scaffold volume fraction should always take values between zero and one in order to be reasonably interpreted as a volume fraction. More restrictive, $\rho$ should even be bounded away from zero and one.} In formulas, we denote by $\mathcal{E}$ the elastic energy
\begin{equation*}
    \mathcal{E}(y,\rho)(t) = \frac12\int_\Omega \mathbb{C}(\rho(x),\sigma(t),b(t,x))\varepsilon(u(t,x)):\varepsilon(u(t,x))\mathrm{d}x
\end{equation*}
where $y=(u,a_1,a_2,c,b)$ is the state variable. The minimization problem is the task to find
\begin{equation}\label{equation:repeat_optimization_problem_simulation_section}
    \rho \in \operatorname{argmin}\left[ \max_{t\in I}\mathcal{E}(y,\rho)(t) \right], \quad \text{subjected to }e(y,\rho)=0\text{ and }\rho\in P,
\end{equation}
where $P$ encodes that $\rho$ is bounded away from zero and one. Numerically, we replace the temporal maximum by an $L^p(I)$ norm (with, e.g.,\ $p=5$) to smoothly approximate it. The pointwise constraint $\rho\in P$ is treated by a soft penalty and $e(y,\rho)=0$ by the adjoint method.
\subsection{Numerical Implementation}
The numerical realization of the PDE constrained optimization problem is based on the adjoint approach, see for instance \cite{hinze2008optimization} for a derivation of the method. This means that the constraint $e(y,\rho)=0$ (in the notation of Section~\ref{section:mathematical_formulation}) is parametrized by the solution operator $\rho\mapsto \phi(\rho)$ satisfying $e(\phi(\rho),\rho)=0$ eliminating the constraint in the optimization. Computing the derivative of the reduced objective with respect to $\rho$ yields an adjoint equation that is structurally similar to the state equations \eqref{equation:state_equations_optimal_control_setting}. Having access to the derivative of the reduced objective, we use an $L^2(\Omega)$ gradient flow in order to solve the optimization problem \eqref{equation:repeat_optimization_problem_simulation_section}. As this leads to reasonable results, more sophisticated optimization algorithms were not deemed necessary.

We use the Computational Geometry Algorithms Library CGAL (\cite{boissonnat2000triangulations}) to generate tetrahedral meshes for the spatial resolution of diffusion and elasticity via P1 finite elements in both the state and adjoint equation. The meshes used in our simulations consist of roughly $40$k tetrahedrons. The time dependence and couplings in the equations are treated by a semi-implicite ansatz, using only the quantities explicitly that are not available at a current time step due to the couplings of the equations. The ODEs are solved on every element separately, yielding a spatially constant approximation of their solution. Due to the comparatively simple structure of the time dependent equations, a coarse time stepping can be employed with one temporal increment corresponding to one week of the regeneration process.

\subsection{Discussion}
\begin{figure}
    \centering
    \begin{subfigure}{0.49\linewidth}
    \includegraphics[width=\linewidth]{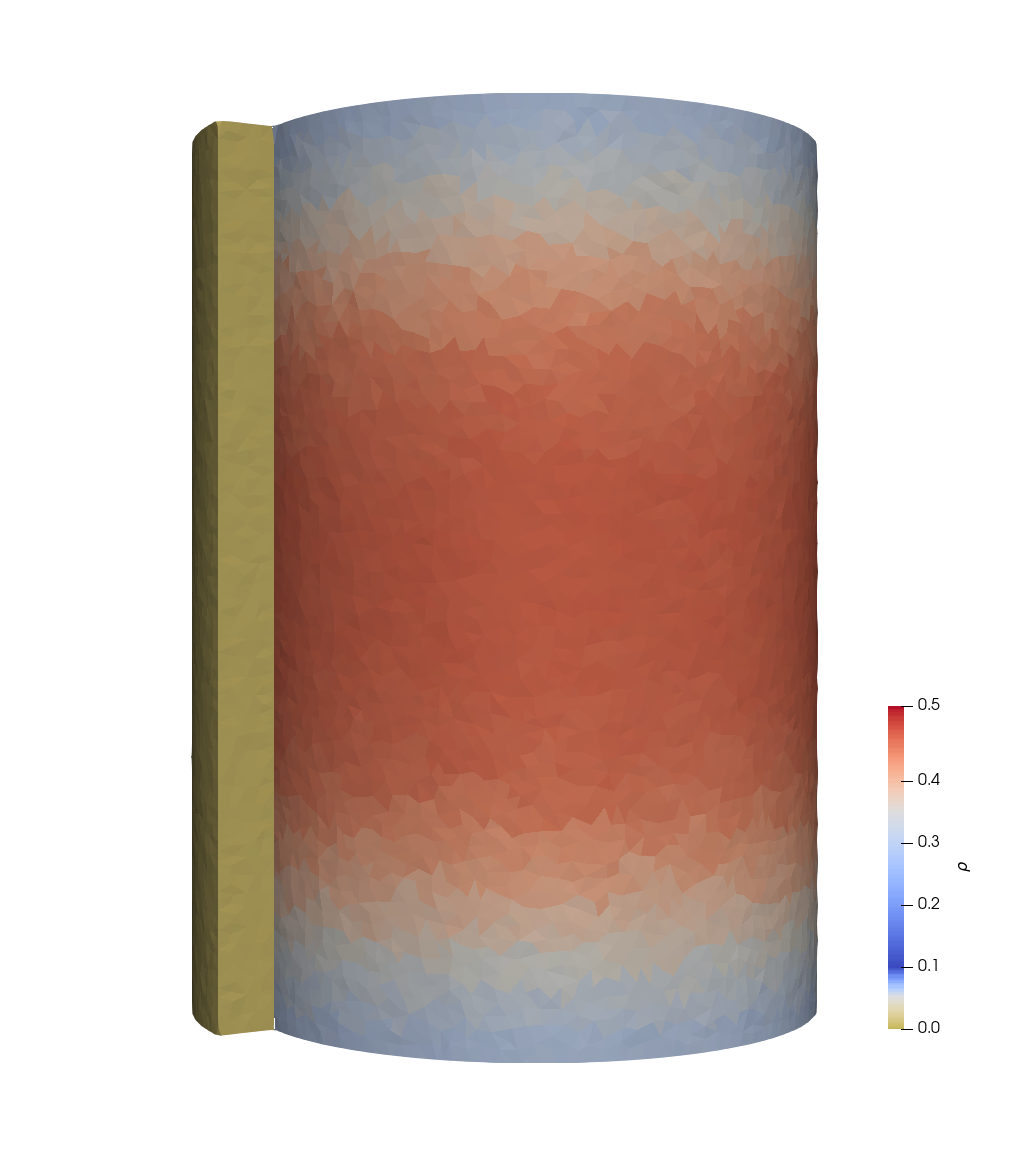}
     \subcaption{Scaffold Architecture A}
    \end{subfigure}
    \begin{subfigure}{0.49\linewidth}
    \includegraphics[width=\linewidth]{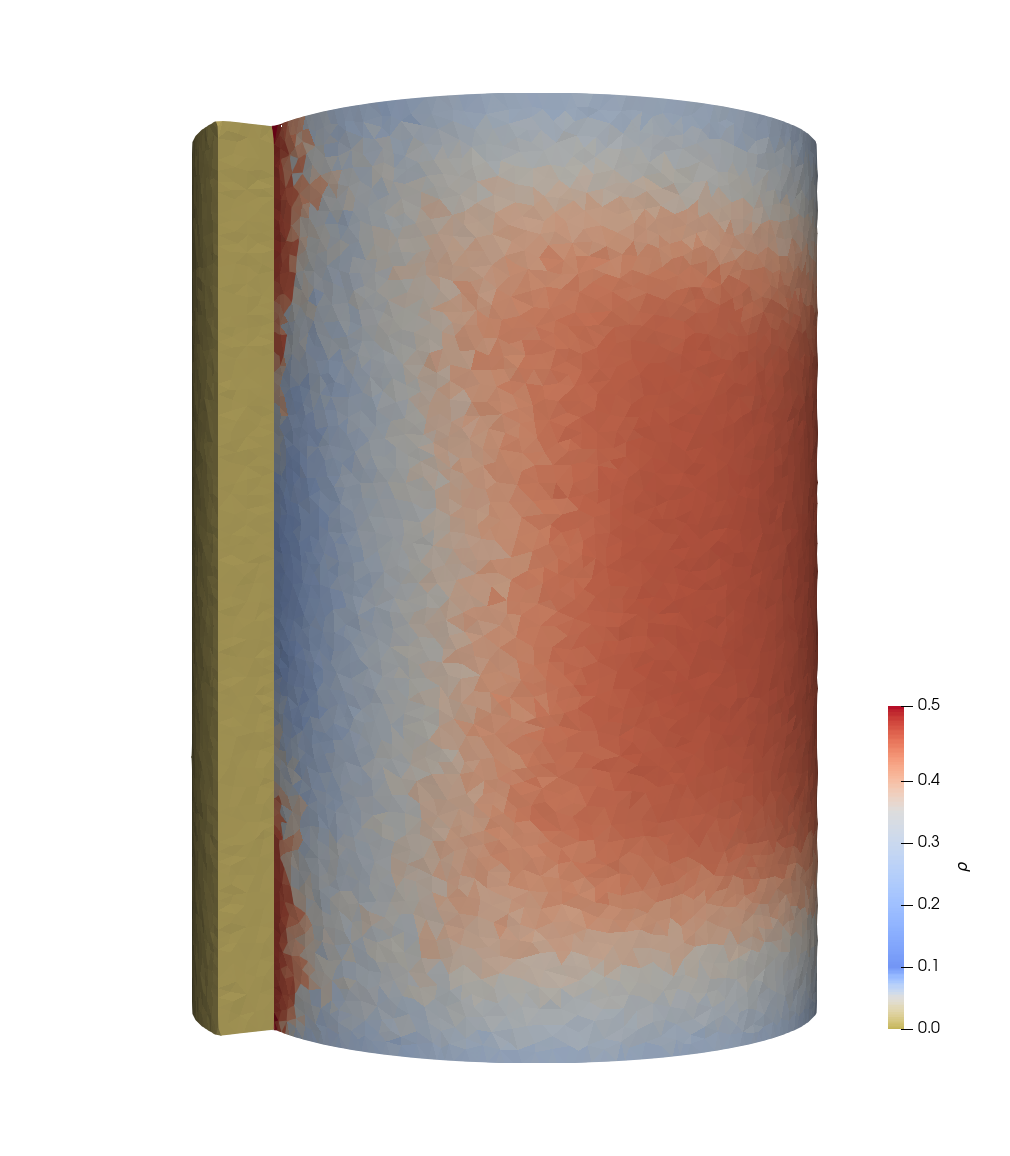}
     \subcaption{Scaffold Architecture B}
    \end{subfigure}
    \caption{Two optimized scaffold densities for different mechanical environments. Architecture A is the result of a one dimensional optimization routine which has afterwards been transferred to a three dimensional setting whereas in architecture B the metal fixateur is included in the optimization routine in a three dimensional model. }
    \label{fig:optimized_scaffolds}
\end{figure}
In figure~\ref{fig:optimized_scaffolds} we display two optimized scaffold densities. Architecture A corresponds to the outcome of an essentially one dimensional experiment setup. To produce architecture A, the fixateur (marked in gold) is excluded in the computations and a compressive softload is applied on the top and bottom of the cylindrical domain. This parallels the optimization routine proposed in \cite{poh2019optimization} and produces a qualitatively similar result. The scaffold architecture B is obtained from an optimization routine including the fixateur and thus takes into account the drastic change in mechanical environment introduced by the fixating element. Using external fixation, the mechanical stimulus is almost absent in the vicinity of the fixateur.  Naturally, this influences the scaffold optimization and an important merit of a three dimensional model is the ability to resolve these stress shielding effects and adapt the architecture of an optimal scaffold accordingly.

The architecture A in figure \ref{fig:optimized_scaffolds} depicts a scaffold with a higher density in the middle region. A reasonable outcome, as regenerated bone grows back at the scaffold ends where it is attached to the intact bone tissue. Therefore, the central scaffold region needs to maintain structural integrity for a longer time by itself. The overall shape is very similar to the results obtained by \cite{poh2019optimization} with a one dimensional model which is not surprising as our experiment is essentially one dimensional.

The architecture B, corresponding to the experiment including the fixateur depicted in the right of figure \ref{fig:optimized_scaffolds}, shows a considerably different distribution. A higher density in the central part is favorable for the same reason as in the experiment excluding the fixateur, however, in vicintiy of the stiff metal plate a comparatively low scaffold density is predicted. High porosity in this region of the scaffold is beneficial as it increases the mechanical stimulus due to reduced stability and enhances vascularization\footnote{In our model vascularization is resolved through the diffusion of bio-active molecules.}. Both effects promote bone ingrowth in the region close to the fixateur. 

To illustrate the benefits with respect to stress shielding of the scaffold architecture B over scaffold architecture A, compare to figure~\ref{fig:optimized_scaffolds}, we use the architecture A in a numerical experiment \emph{including external fixation}. We then compare the strain distributions for the two architectures. Figure~\ref{fig:initial_strain_distribution} shows the strain magnitude distributions at the initial time-point when no bone has regenerated yet. We clearly observe that architecture B mitigates stress shielding in the vicinity of the fixateur in comparison to architecture A. This trend is sustained two month in the regeneration process, as can be observed in figure~\ref{fig:2months_strain_distribution}.  We remark that the reduction of stress-shielding is not directly part of the objective function with respect to which the optimization is carried out. Rather, this effect is an implicit favorable consequence of the objective function~\eqref{equation:repeat_optimization_problem_simulation_section} that advocates for its usage in scaffold design optimization. 

\begin{figure}
    \centering
    \begin{subfigure}{0.48\linewidth}
    \includegraphics[width=\linewidth]{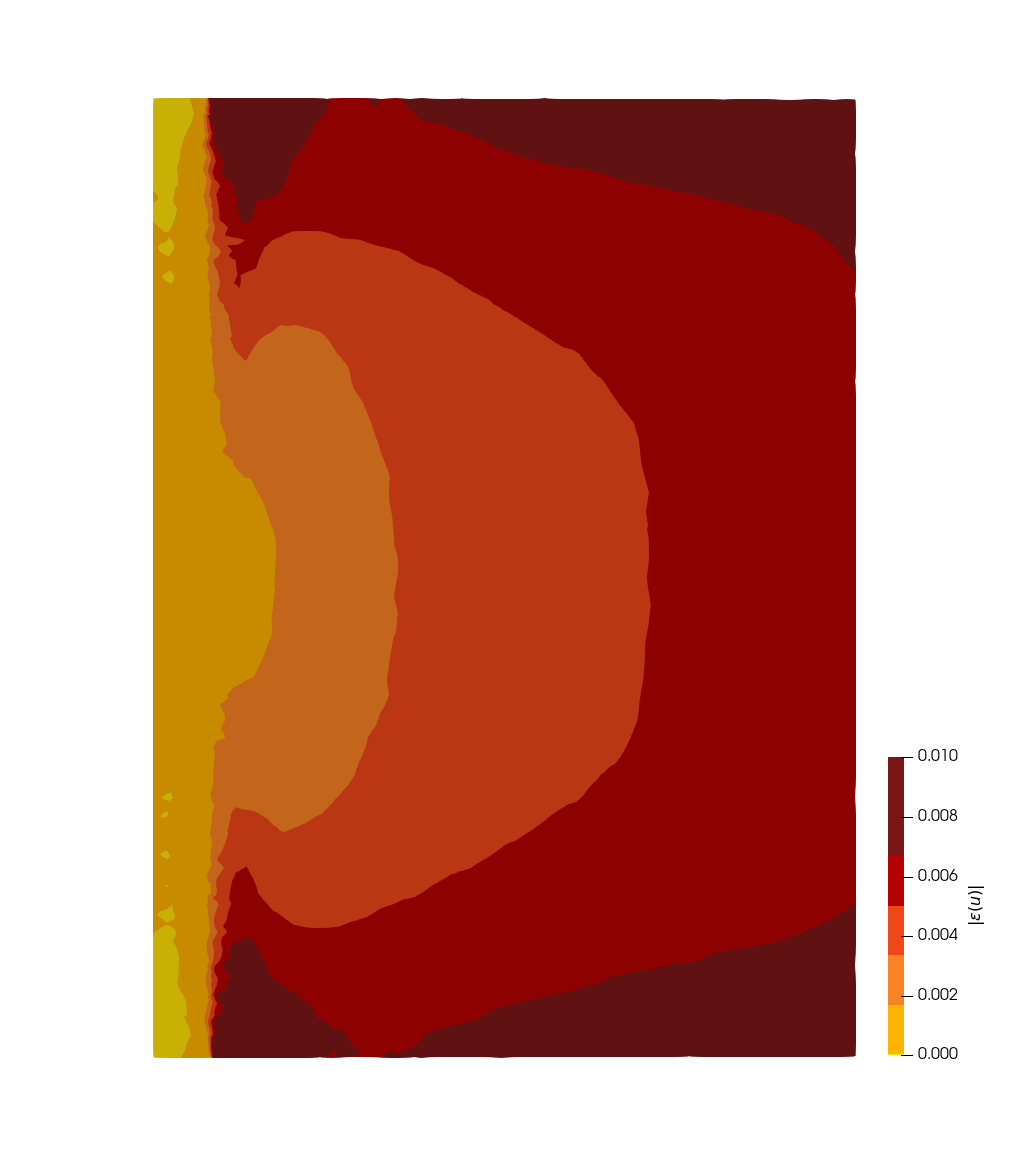}
    \end{subfigure}
    \begin{subfigure}{0.48\linewidth}
    \includegraphics[width=\linewidth]{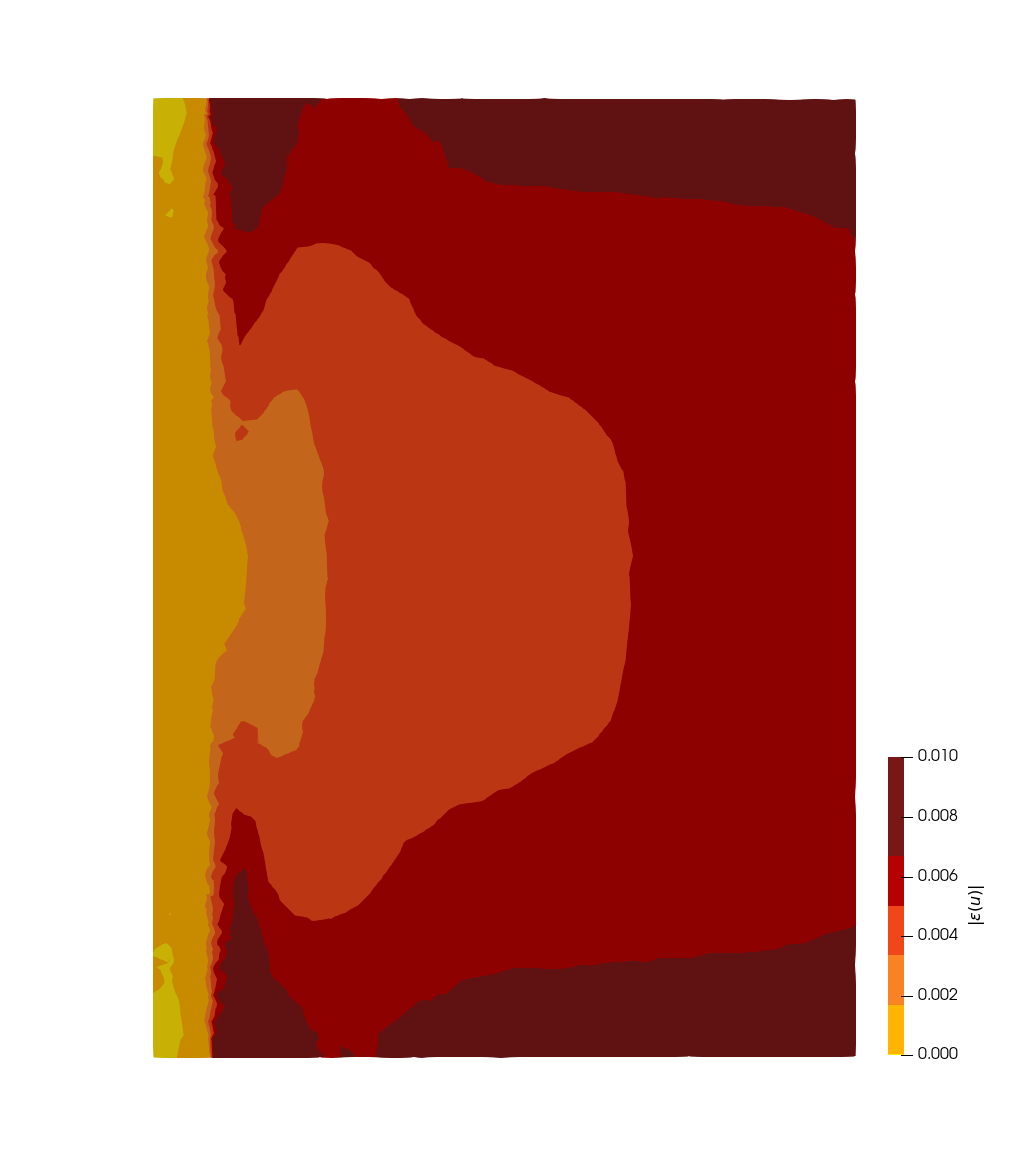}
    \end{subfigure}
    \caption[Comparison of Strain Magnitudes before Bone Regeneration]{The strain magnitude distributions for the scaffold architectures displayed in figure~\ref{fig:optimized_scaffolds} are compared in a vertical cut plane before the healing process, i.e., when no bone has regenerated yet. The left picture corresponds to architecture A and the right figure to architecture B of figure~\ref{fig:optimized_scaffolds}.}
    \label{fig:initial_strain_distribution}
\end{figure}

\begin{figure}
    \begin{subfigure}{0.48\linewidth}
    \includegraphics[width=\linewidth]{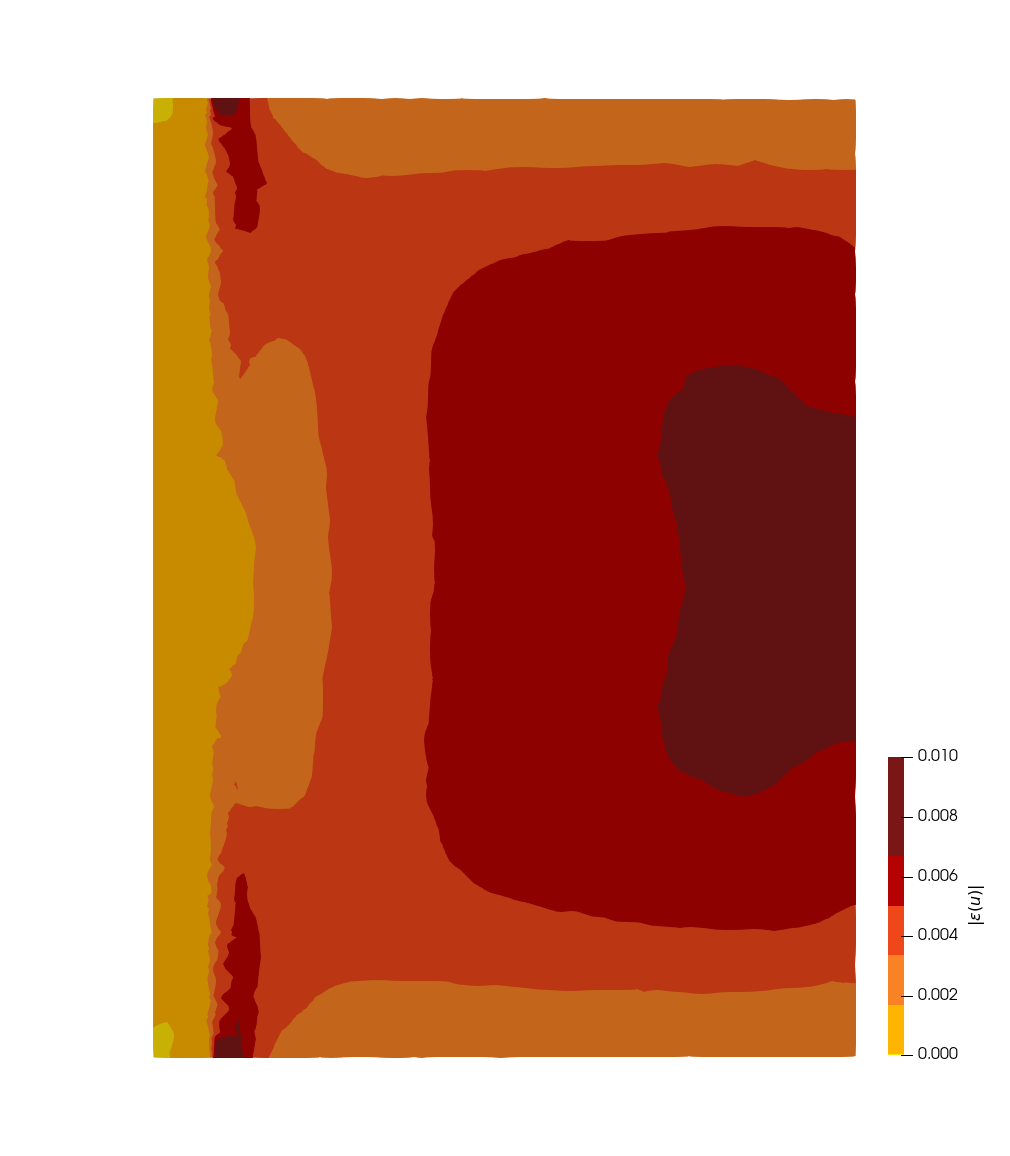}
    \end{subfigure}
    \begin{subfigure}{0.48\linewidth}
    \includegraphics[width=\linewidth]{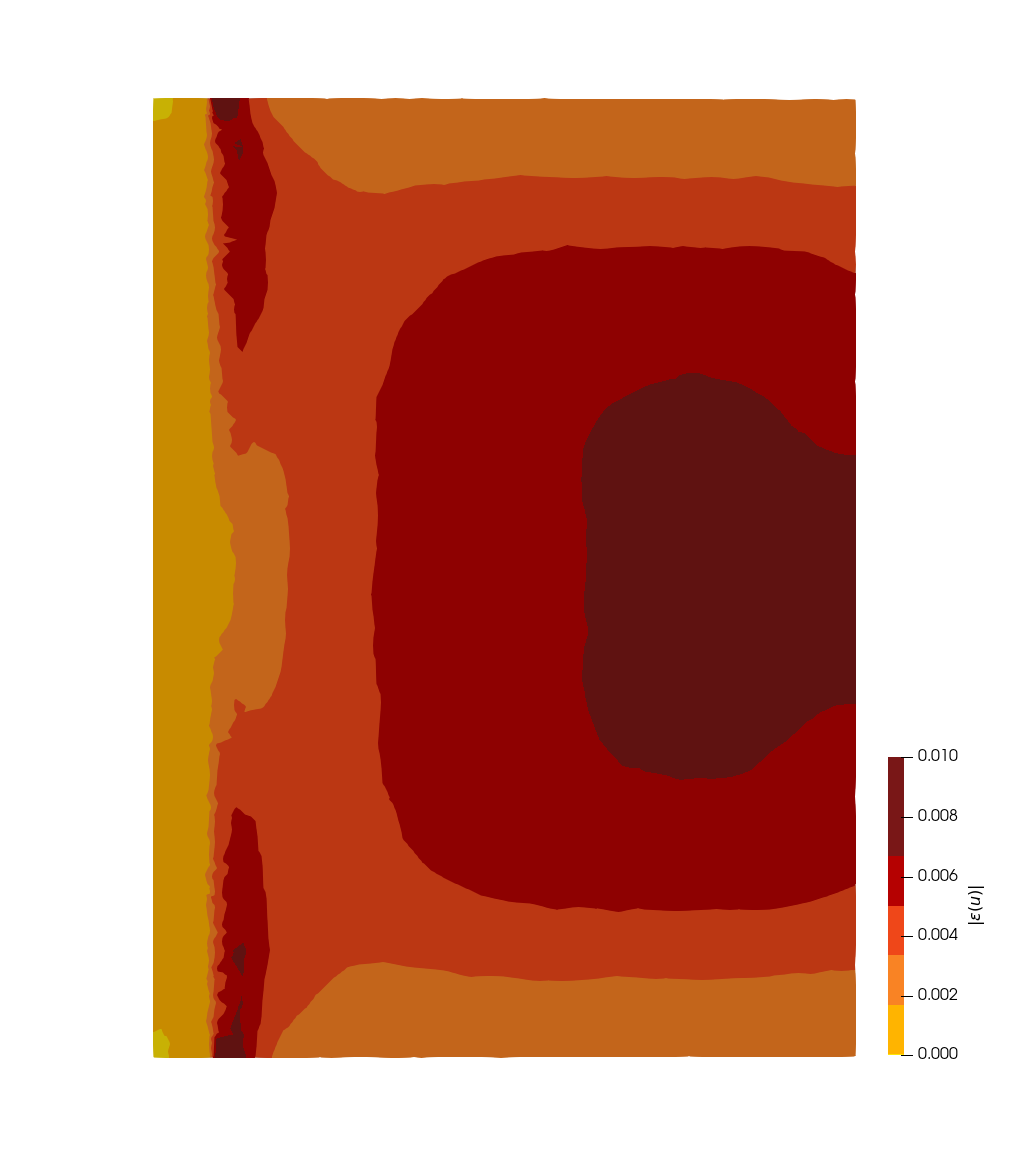}
    \end{subfigure}
    \caption[Comparison of Strain Magnitudes after two Months of Bone Regeneration]{The strain magnitude distributions for the scaffold architectures displayed in figure~\ref{fig:optimized_scaffolds} are compared in a vertical cut plane two months in the healing process. The left picture corresponds to architecture A and the right figure to architecture B of figure~\ref{fig:optimized_scaffolds}.}
    \label{fig:2months_strain_distribution}
\end{figure}

\section{Conclusions and Future Research}
We analyzed a three dimensional, homogenized model for bone growth in the presence of a porous, bio-resorbable scaffold and considered the associated problem of optimal scaffold design. This leads to a PDE constrained optimization problem for which we proved the existence of an optimal control, i.e.,\ an optimal scaffold density distribution. We presented proof-of-concept numerical experiments illustrating the benefits of a three dimensional optimization routine. For future work, we propose to use the computational model in detailed numerical simulations and to study the optimized scaffold architectures in vivo.

\bibliographystyle{apalike}
\bibliography{references}

\appendix
\section{Proofs of the Main Results}\label{section:proofs_of_the_main_results}
\begin{lemma}[$C^0(I,H^1(\Omega))$ bound for $u$]\label{lemma:bound_for_u_k}
    Let $\mathbb{C} \in C^0(I,L^\infty(\Omega,\mathcal{L}(\mathcal{M}_s)))$ be uniformly elliptic with ellipticity constant $\lfloor \mathbb{C} \rfloor$ independent of $t\in I$ and $x\in\Omega$, i.e., it holds
    \begin{equation*}
        \mathbb{C}(t,x)M:M \geq \lfloor \mathbb{C} \rfloor\lvert M \rvert^2, \quad\text{for all }M\in \mathcal{L}(\mathcal{M}_s)\text{ and }(t,x)\in I\times\Omega.
    \end{equation*}
    Furthermore, let $f \in C^0(I,H^1_D(\Omega)^*)$ be a fixed right-hand side. Then the unique solution $u\in L^2(I,H^1_D(\Omega))$ to 
    \begin{equation}\label{equation:abstract_elastic_equation}
        \iint \mathbb{C}\varepsilon(u):\varepsilon(\cdot)\mathrm dx \mathrm dt = \int_I \langle f,\cdot \rangle_{H_D^{1}(\Omega)}\mathrm dt \quad \text{in } L^2(I,H^1_D(\Omega))^*
    \end{equation}
    is a member of the space $C^0(I,H^1_D(\Omega))$ and satisfies
    \begin{equation*}
        \lVert u \rVert_{C^0(I,H^1(\Omega))} \leq C\left( \lfloor \mathbb{C}\rfloor, C_{\text{Korn}} \right)\cdot \lVert  f \rVert_{C^0(I,H^1_D(\Omega)^*)}.
    \end{equation*}
\end{lemma}
\begin{proof}
    The equation \eqref{equation:abstract_elastic_equation} implies that $u$ satisfies almost everywhere in $I$
    \begin{equation*}
        \underbrace{\int_\Omega \mathbb{C}(t)\varepsilon(u(t)):\varepsilon(\cdot)\mathrm dx}_{\eqqcolon \mathcal{T}_tu(t)} = f(t) \quad\text{in }H^1_D(\Omega)^*
    \end{equation*}
    upon applying the isometry $L^2(I,H^1_D(\Omega))^* \to L^2(I,H^1_D(\Omega)^*)$ to both sides of the equation. Clearly, testing with $u(t)$ yields, inferring Korn's inequality,
    \begin{equation*}
        \lfloor \mathbb{C} \rfloor C_{\text{Korn}} \cdot \lVert u \rVert^2_{H^1_D(\Omega)} \leq \lfloor \mathbb{C} \rfloor \cdot \lVert \varepsilon(u) \rVert_{L^2(\Omega)}^2 \leq \lVert f(t) \rVert_{H^1_D(\Omega)^*}\lVert u(t) \rVert_{H^1_D(\Omega)}. 
    \end{equation*}
    Hence, 
    \begin{equation*}
        \lVert u(t) \rVert_{H^1_D(\Omega)} 
        \leq
        \left( \lfloor \mathbb{C} \rfloor C_{\text{Korn}} \right)^{-1} \lVert f(t) \rVert_{H^1_D(\Omega)^*}
        \leq
        \left( \lfloor \mathbb{C} \rfloor C_{\text{Korn}} \right)^{-1} \lVert f \rVert_{C^0(I,H^1_D(\Omega)^*)},
    \end{equation*}
    meaning that the $H^1(\Omega)$ bound on $u(t)$ is independent of $t\in I$. To show that $u$ is continuous in time, we compute for $t,s \in I$
    \begin{equation*}
        f(t) - f(s) = \mathcal{T}_tu(t) - \mathcal{T}_su(s) = \mathcal{T}_t(u(t)-u(s)) + \mathcal{T}_t(u(s)) - \mathcal{T}_su(s).
    \end{equation*}
    Using the coercivity of $\mathcal{T}_t$ we find
    \begin{equation*}
        \lVert u(t) - u(s) \rVert_{H^1(\Omega)} \leq \frac{1}{\lfloor \mathbb{C} \rfloor C_{\text{Korn}}}\left[ \lVert f(t) - f(s) \rVert_{H^1(\Omega)^*} + \lVert \mathcal{T}_tu(s) - \mathcal{T}_su(s) \rVert_{H^1_D(\Omega)^*} \right]
    \end{equation*}
    By the assumption $f\in C^0(I,H^1_D(\Omega)^*)$ it is clear that the first term above tends to zero when $|t-s|\to 0$. It remains to estimate
    \begin{align*}
        \lVert \mathcal{T}_tu(s) - \mathcal{T}_su(s) \rVert_{H^1_D(\Omega)^*} 
        & \leq 
        \sup_{\lVert \varphi \rVert_{H^1_D(\Omega)}\leq 1}\int_\Omega \left[ \mathbb{C}(t) - \mathbb{C}(s) \right]\varepsilon(u(s)):\varepsilon(\varphi)\mathrm dx
        \\
        & \leq 
        \lVert \mathbb{C}(t) - \mathbb{C}(s) \rVert_{L^\infty(\Omega, \mathcal{M}_s)}\lVert u(s) \rVert_{H^1_D(\Omega)}.
    \end{align*}
    The time-independent bound on $\lVert u(s) \rVert_{H^1_D(\Omega)}$ and the continuity assumption on $\mathbb{C}$ imply the assertion.
\end{proof}
\begin{lemma}[Equi-Continuity]\label{lemma:equi_continuity_of_uk} Assume $(\rho_k)\subset P$ is any sequence, $(b_k)\subset W_{\rho_k}$ is an equi-continuous sequence in $C^0(I,C^0(\Omega))$ and $(f_k)$ is a equi-continuous and bounded sequence in $C^0(I,H^1_D(\Omega)^*)$. Assume that $\mathbb{C}(\rho_k,\sigma, b_k)$ satisfies the assumption \ref{section:setting_optimal_control}, i.e., in particular, it holds 
\begin{equation}
    \lVert \mathbb{C}(\rho_k,\sigma, b_k)(t) - \mathbb{C}(\rho_k, \sigma, b_k)(s) \rVert_{L^\infty(\Omega, \mathcal{L}(\mathcal{M}_s))} \leq C \lVert b_k(t) - b_k(s) \rVert_{C^0(\Omega)}
\end{equation}
for a constant $C$ that does not depend on the data $\rho_k\in P$ and $b_k\in W_{\rho_k}$ and $t\in I$. Denote by $u_k$ the unique solution of
\begin{equation*}
    \iint \mathbb{C}(\rho_k,\sigma,b_k)\varepsilon(u_k):\varepsilon(\cdot)\mathrm dx \mathrm dt = \int_I \langle f_k,\cdot \rangle_{H^1_D(\Omega)}\mathrm dt \quad \text{in }L^2(I,H^1_D(\Omega))^*.
\end{equation*}
Then, $(u_k)$ lies in $C^0(I,H^1_D(\Omega))$ and is equi-continuous in this space.
\end{lemma}
\begin{proof}
    We are in situation of Lemma \ref{lemma:bound_for_u_k}, hence we know that $u_k$ is a member of the space $C^0(I,H^1_D(\Omega))$ and we need only to establish the equi-continuity. To this end, repeating the equations in Lemma \ref{lemma:bound_for_u_k} for $u_k$ instead of $u$ we arrive at
    \begin{align*}
        \lVert u_k(t) - u_k(s) \rVert_{H^1(\Omega)} 
        &\leq
        \frac{1}{\lfloor \mathbb{C}_k \rfloor C_{\text{Korn}}}\left[  \lVert f_k(t) - f_k(s) \rVert_{H^1(\Omega)^*} + \lVert \mathbb{C}_k(t) - \mathbb{C}_k(s) \rVert_{L^\infty(\Omega,\mathcal{L}(\mathcal{M}_s))} \lVert u_k(s) \rVert_{H^1(\Omega)} \right]
        \\&\leq
        \frac{1}{\lfloor \mathbb{C}_k \rfloor C_{\text{Korn}}}\left[  \lVert f_k(t) - f_k(s) \rVert_{H^1(\Omega)^*} + C\lVert b_k(t) - b_k(s) \rVert_{C^0(\Omega)} \right]
    \end{align*}
    as $\lVert u_k(t) \rVert_{H^1(\Omega)}$ is bounded uniformly in $k\in\mathbb{N}$ and $s\in I$ by Lemma \ref{lemma:bound_for_u_k} through the boundedness we assumed for $(f_k)$. Then, we infer the equi-continuity of $(f_k)$ and $(b_k)$ to derive it for $(u_k)$.
\end{proof}
The following lemma summarizes the main result of \cite{haller2019higher}. We restrict ourselves to the generality necessary needed for our application, which however, is not the most general situation. We refer the reader to \cite{haller2019higher} for a relaxation concerning boundary regularity, regularity of coefficients and the differential operator. 
\begin{lemma}[Higher Regularity for Elliptic Systems]\label{lemma:higher_regularity_elliptic_equation}
    Let $\mathbb{C} \in L^\infty(\Omega, \mathcal{L}(\mathcal{M}_s))$ be uniformly elliptic, i.e., there exists $\lfloor \mathbb{C} \rfloor > 0$ such that
    \begin{equation*}
        \mathbb{C}M:M \geq \lfloor \mathbb{C} \rfloor |M|^2, \quad \text{for all }M\in\mathcal{M}_s.
    \end{equation*}
    Assume that $\mathbb{C}_{ijkl}\in C^\alpha(\Omega)$ for a fixed but arbitrary small $\alpha >0$. Then, there exists $\theta = \theta(\alpha) > 0$ such that for every $f\in H^{1-\theta}(\Omega)^*$ the solution $u\in H^1_D(\Omega)$ to 
    \begin{equation*}
        \int_\Omega \mathbb{C}\varepsilon(u):\varepsilon(\cdot)\mathrm dx = f \quad \text{in }H^1_D(\Omega)^*
    \end{equation*}
    is in fact a member of $H^{1+\theta}(\Omega)$ and we can estimate
    \begin{equation*}
        \lVert u \rVert_{H^{1+\theta}(\Omega)} \leq C\lVert \mathbb{C} \rVert_{C^\alpha(\Omega)}\lVert f \rVert_{H^{1 - \theta}_D(\Omega)^*},
    \end{equation*}
    where $C$ does not depend on the concrete form of $\mathbb C$.
\end{lemma}
\begin{proof}
    This follows from Theorem 1 and Lemma 1 in \cite{haller2019higher}. 
\end{proof}
The last result we need to establish the relative compactness of $(u_k)$ in $C^0(I,H^1_D(\Omega))$ is -- little surprisingly -- a vector valued version of the Arzel\`a-Ascoli Theorem which we recall here for convenience.
\begin{theorem}[Characterization of Relative Compactness in $C^0(K,X)$ Spaces]\label{theorem:arzela_ascoli_vector_valued}
    Let $X$ be a Banach space and $K$ a compact metric space. Then a set $\mathcal{F}\subset C^0(K,X)$ is relatively compact if and only if the following two conditions hold:
    \begin{itemize}
        \item [(i)] The set $\mathcal{F}$ is equi-continuous, that is, for all $t\in K$ and all $\varepsilon >0$ there exists a neighborhood $U(t)\subset K$ such that 
        \begin{equation*}
            \sup_{u\in \mathcal{F}}\lVert u(t) - u(s) \rVert_{X} \leq \varepsilon\quad\text{for all }s\in U(t).
        \end{equation*}
        \item[(ii)] For all $t\in K$ the set
        \begin{equation*}
            \{ u(t)\mid u\in \mathcal{F} \} \subset X
        \end{equation*}
        is relatively compact.
    \end{itemize}
\end{theorem}

The focus of the next lemma lies on the a priori estimates for linear parabolic equations. 
\begin{lemma}[A Priori Estimate for Parabolic Evolution Equations]\label{lemma:bound_for_a_k}
    Let $(i,X,H)$ be a Gelfand triple, $M:X\to X^*$ a linear coercive operator with coercivity constant $\lfloor M \rfloor$, i.e., it holds 
    \begin{equation*}
        \langle M a, a\rangle_X \geq \lfloor M \rfloor \lVert a \rVert_X^2, \quad \text{for all }a\in X.
    \end{equation*}
    Let $I = [0,T]$ denote a time interval and $f\in L^2(I,X^*)$ a fixed right-hand side. Then there exists a unique solution $a\in H^1(I,X,X^*)$ to 
    \begin{equation*}
        \int_I \langle d_ta, \cdot \rangle \mathrm dt + \int_I\langle Ma, \cdot \rangle_X\mathrm dt = \int_I \langle f,\cdot \rangle_X \mathrm dt, \quad \text{in } L^2(I,X)^*
    \end{equation*}
    Furthermore, the norm of the solution $a$ can be estimated by
    \begin{equation}\label{equation:a_priori_estimate_evolution_problems}
        \lVert a \rVert_{H^1(I,X,X^*)} \leq C\left( \lVert M \rVert_{\mathcal{L}(X,X^*)}, \lfloor M \rfloor^{-1}  \right)\cdot \left( \lVert a(0) \rVert_{H} + \lVert f \rVert_{L^2(I,X^*)}\right),
    \end{equation}
    with $C$ being monotonously increasing in  $\lVert M \rVert_{\mathcal{L}(X,X^*)}$ and $\lfloor M \rfloor^{-1}$.
\end{lemma}
\begin{proof}
    We establish only the estimate \eqref{equation:a_priori_estimate_evolution_problems}, the existence of a solution is the well known maximal regularity result of J.\ L.\ Lions, see for instance \cite[Part II, Section 6]{ern2013theory}. To derive the estimate, we note that by the natural isometry $L^2(I,X^*) = L^2(I,X)^*$ the function $a$ satisfies a pointwise almost-everywhere equation in $X^*$, namely
    \begin{equation*}
        d_ta(s) + M(a(s)) = f(a(s))
    \end{equation*}
    which, at time $s\in I$, we can test with $a(s) \in X$ and integrate from $0$ to $t$. Then, we apply the partial integration formula for Gelfand triples and estimate using the coercivity of $M$ and Young's inequality
    \begin{align*}
        \frac12 \lVert a(t) \rVert_H^2 + \lfloor M \rfloor \int_0^t\lVert a(s) \rVert^2_X\mathrm ds 
        &\leq 
        \frac12 \lVert a(0) \rVert_H^2 + \lVert f \rVert_{L^2(I,X^*)}\lVert a \rVert_{L^2(I,X)}
        \\&\leq
        \frac12 \lVert a(0) \rVert_H^2 + \frac{1}{2\lfloor M \rfloor}\lVert f \rVert^2_{L^2([0,t],X^*)} + \frac{\lfloor M \rfloor}{2} \lVert a \rVert^2_{L^2([0,t],X)},
    \end{align*}
    which leads to
    \begin{equation*}
        \frac12 \lVert a(t) \rVert_H^2 + \frac{\lfloor M \rfloor}{2} \int_0^t\lVert a(s) \rVert^2_X\mathrm ds
        \leq 
        \frac12 \lVert a(0) \rVert_H^2 + \frac{1}{2\lfloor M \rfloor}\lVert f \rVert^2_{L^2([0,t],X^*)}.
    \end{equation*}
    We get by estimating the terms of the left-hand side separately and taking the supremum over $t\in I$ both
    \begin{gather*}
        \lVert a \rVert_{C^0(I,L^2(\Omega))}^2 \leq  \lVert a(0) \rVert_H^2 + \frac{1}{\lfloor M \rfloor}\lVert f \rVert^2_{L^2(I,X^*)}
        \quad \text{and} \quad 
        \lVert a \rVert_{L^2(I,X)} \leq \lVert a(0) \rVert_H^2 + \frac{1}{\lfloor M \rfloor^2}\lVert f \rVert^2_{L^2(I,X^*)}.
    \end{gather*}
    To estimate the $L^2(I,X)^*$ norm of $d_ta$, we use that $a$ is the solution of the parabolic equation to estimate
    \begin{align*}
        \lVert d_ta \rVert_{L^2(I,X)^*} 
        &=
        \sup_{\lVert \varphi \rVert_{L^2(I,X)\leq 1}} \int_I \langle d_ta, \varphi  \rangle_X\mathrm dt
        \\ & \leq
        \sup_{\lVert \varphi \rVert_{L^2(I,X)\leq 1}} \left[ \int_I |\langle Ma,\varphi\rangle_X|\mathrm dt + \int_I|\langle f,\varphi \rangle_X| \mathrm dt \right]
        \\ & \leq 
        \lVert M \rVert_{\mathcal{L(X,X^*)}}\lVert a \rVert_{L^2(I,X)} + \lVert f \rVert_{L^2(I,X^*)}.
    \end{align*}
    If we infer the previous estimates for $a$ in $L^2(I,X)$ norm, we can bound $d_ta$ in $L^2(I,X)^*$ norm. Combining the considerations for $a$ and $d_ta$ lets us bound the $H^1(I,X,X^*)$ as desired.
\end{proof}
\begin{lemma}[$L^p(I,C^\alpha(\Omega))$ Bound for $(a_k)$]\label{lemma:the_bound_for_a} 
    Assume $\Omega\subset \mathbb{R}^d$ with $d=1,2,3$ and     $\partial\Omega = \Gamma_N\cup\Gamma_D$ where $\Omega \cup \Gamma_N$ is Gr\"oger regular. Let $D\in L^\infty(\Omega,\mathcal{M}_s)$ be uniformly elliptic with ellipticity constant $\lfloor D \rfloor >0$, $k > 0$, $f\in     L^p(I,L^2(\Omega))$ for a fixed $p > 2$ and $a_0 \in L^\infty(\Omega)$ some essentially bounded initial condition. Then there exists $\alpha = \alpha(p)\in(0,1)$ independent of $D$ and $f$ such that the solution $a \in     H^1(I,H^1_D(\Omega),H^1_D(\Omega)^*)$ to 
    \begin{align*}
        \int_I \langle d_ta,\cdot \rangle_{H^1_D(\Omega)}\mathrm dt + \iint D\nabla a\nabla \cdot + k a\cdot \mathrm dx \mathrm dt &= \iint f\cdot \mathrm dx\mathrm dt
        \\
        a(0) = a_0
    \end{align*}
    is a member of $L^p(I,C^\alpha(\Omega))$ and satisfies the estimate
    \begin{equation*}
        \lVert a \rVert_{L^p(I,C^\alpha(\Omega))} \leq C\left(p,\lfloor D \rfloor, \lVert D \rVert_{L^\infty(\Omega,\mathcal{M}_s)}\right) \lVert f \rVert_{L^p(I,L^2(\Omega))}.
    \end{equation*}
\end{lemma}
\begin{proof}
    The proof of this Lemma exceeds the scope of this manuscript and is the main result of \cite{dondl2021regularity}.
\end{proof}
\begin{remark}
    We comment on some of the aspects leading to the complexity of the proof of lemma \ref{lemma:the_bound_for_a}.
    \begin{itemize}
        \item [(i)] The mixed boundary conditions, rough coefficients and the jump initial condition prevents the standard theory from being applicable. If it wasn't for this roughness, an $L^2(I,H^2(\Omega))$ result could be derived by standard theory, see for instance \cite{evans1998partial}.
        \item[(ii)] Even invoking the theory of abstract parabolic equations as described in \cite{amann1995linear} does only almost suffice. In fact, combining the results in \cite{amann1995linear} with \cite{haller2009holder} yields $L^2(I,C^\alpha(\Omega))$ regularity only if $a_0$ lies in a suitable trace space for initial conditions. The trace space in this case is $H^1_D(\Omega)$ and not $L^\infty(\Omega)$.
        \item[(iii)] The strategy to prove Lemma \ref{lemma:the_bound_for_a} is therefore to treat the cases $f = 0$, $a(0) = a_0$ and $f = f$, $a(0) = 0$ separately and then to use the superposition principle for linear equations. For details we refer to \cite{dondl2021regularity}.
    \end{itemize}
\end{remark}
We treat now the cell ODE. We need to establish that a solution in $W^{1,2}(I,C^\alpha(\Omega))$ exists and is suitably bounded in the data. We have already access to the fact that a long-time solution in $W^{1,2}(I,C^0(\Omega))$ exists, hence the crucial part is to control the H\"older norm of this solution. This can be done by accessing the solution $c$ through its formulation as a fixed-point and then estimating its H\"older norm if suitable regularity for the data is given.

\begin{lemma}\label{lemma:holder_bound_c_k}
    Let $a_1$ and $a_2$ be functions in $L^4(I,C^\alpha(\Omega))$ with $a_1, a_2 \geq 0$. Assume that $\rho \in C^\alpha(\Omega)$ satisfies $0 < \rho < 1$ and $k_6$ and $k_7$ are positive constants. Then there exists a solution $c \in W^{1,2}(I,C^0(\Omega))$ to the equation
    \begin{equation*}
        d_tc = k_6a_1a_2(1+k_7c)\left( 1 - \frac{c}{1-\rho} \right), \quad c(0) = 0
    \end{equation*}
    with $0\leq c \leq 1$. Furthermore, we can control the $\alpha$-H\"older seminorm of $c$ in the following way
    \begin{equation*}
        \lfloor c(t) \rfloor_\alpha \leq C\left( \lVert a_1 \rVert_{L^2(I,C^\alpha(\Omega))}, \lVert a_2 \rVert_{L^2(I,C^\alpha(\Omega))}, \lVert \rho \rVert_{C^\alpha(\Omega)}  \right),
    \end{equation*}
    with the constant $C$ being monotone in its arguments.
\end{lemma}
\begin{proof}
    The existence of a solution in the space $W^{1,2}(I,C^0(\Omega))$ was already established in Theorem 3.2 in \cite{dondl2021efficient}. We are only concerned with the control over the H\"older seminorm. To simplify notation, we prove the statement for an ODE of the form
    \begin{equation}\label{local}
        dtc = m(1 + c)(1 - \theta c), \quad c(0) = 0
    \end{equation}
    with $m \in L^2(I,C^\alpha(\Omega))$ and $\theta \in C^\alpha(\Omega)$ with $0 < \theta^{-1}(x) < 1$, which implies that the solution $c$ to \eqref{local} takes values in the unit interval, i.e., $c(t,x)\in [0,1]$, see Lemma B.5. The existence of $c$ in $W^{1,2}(I,C^0(\Omega))$ solving \eqref{local} implies upon applying integrating that $c(t)$ is given by
    \begin{equation*}
        c(t) = \int_0^t m(s)(1+c(s))(1-\theta c(s))\mathrm ds,
    \end{equation*}
    with the integral being a $C^0(\Omega)$ valued Bochner integral. As point evaluation at $x\in \overline{\Omega}$ is continuous and linear from $C^0(\Omega)$ to $\mathbb{R}$, it also holds
    \begin{equation*}
        c(t,x) = \int_0^t m(s,x)(1+c(s,x))(1-\theta(x) c(s,x))\mathrm ds.
    \end{equation*}
    We use the above formula and the triangle inequality to estimate
    \begin{align*}
        |c(t,x) - c(t,y)| &\leq \underbrace{\int_0^t|m(s,x)-m(s,y)|\mathrm ds}_{\eqqcolon A} 
        +
        \underbrace{\int_0^t |m(s,x)c(s,x) - m(s,y)c(s,y)| \mathrm ds}_{\eqqcolon B} 
        \\&+
        \underbrace{\int_0^t |m(s,y)\theta(y)c(s,y) - m(s,x)\theta(x)c(s,x)| \mathrm ds}_{\eqqcolon C}
        \\&+
        \underbrace{\int_0^t |m(s,y)\theta(y)c(s,y)^2 - m(s,x)\theta(x)c(s,x)^2| \mathrm ds}_{\eqqcolon D}.
    \end{align*}
    For brevity, we set $\tilde m(t,x) = m(t,x)\theta(x)$. Inferring that $c$ takes values in $[0,1]$, we claim that the above estimate leads to
    \begin{equation}\label{local_holder_estimate}
        |c(t,x) - c(t,y)| \leq \int_0^t \left( 2\lfloor m(s) \rfloor_\alpha + \lVert m(s) \rVert_{C^0(\Omega)}\lfloor c(s) \rfloor_\alpha + 3\lVert \Tilde{m}(s) \rVert_{C^0(\Omega)} \lfloor c(s) \rfloor_\alpha + 2\lfloor \Tilde{m}(s) \rfloor_\alpha \right)|x-y|^\alpha\mathrm ds.
    \end{equation}
    Dividing by $|x-y|^\alpha$ and taking the supremum over pairs $(x,y)\in\overline{\Omega}^2$ with $x\neq y$ we get
    \begin{equation*}
        \lfloor c(t) \rfloor_\alpha 
        \leq
        \int_0^t \underbrace{2 \left( \lfloor m(s) \rfloor_\alpha + \lfloor \tilde m(s) \rfloor_\alpha \right)}_{\eqqcolon \alpha(s)} + \underbrace{\left( \lVert m(s) \rVert_{C^0(\Omega)} + 3\lVert \tilde m(s) \rVert_{C^0(\Omega)} \right)}_{\eqqcolon \beta(s)} \lfloor c(s) \rfloor_\alpha \mathrm ds.
    \end{equation*}
    Hence, by Gr\"onwall's lemma we get
    \begin{equation*}
        \lfloor c(t) \rfloor_\alpha \leq \left( 1 + \lVert \beta \rVert_{L^1(I)}\exp\left(\lVert \beta \rVert_{L^1(I)}\right) \right)\lVert \alpha \rVert_{L^1(I)}
    \end{equation*}
    with
    \begin{align*}
        \lVert \alpha \rVert_{L^1(I)} &\leq 2 \lVert m \rVert_{L^1(I,C^\alpha(\Omega))} + 2 \lVert \tilde m \rVert_{L^1(I,C^\alpha(\Omega))}
        \\
        \lVert \beta \rVert_{L^1(I)} &\leq \lVert m \rVert_{L^1(I,C^0(\Omega))} + 3 \lVert \tilde m \rVert_{L^1(I,C^0(\Omega))}.
    \end{align*}
    As for a bounded intervals the $L^2$ norm dominates the $L^1$ norm, we are done, given we still provide the details of the computations that led to \eqref{local_holder_estimate}. To this end, note that we may estimate $(A)$ by 
    \begin{equation*}
        \int_0^t |m(s,x)-m(s,y)|\mathrm ds \leq \int_0^t \lfloor m(s) \rfloor_\alpha|x-y|^\alpha \mathrm ds.
    \end{equation*}
    Using the triangle inequality and the pointwise properties of $c$, we estimate $(B)$ by
    \begin{align*}
        \int_0^t |m(s,x)c(s,x) - m(s,y)c(s,y)| \mathrm ds 
        &\leq
        \int_0^t |m(s,x)|\lfloor c(s) \rfloor_\alpha |x - y|^\alpha \mathrm ds + \int_0^t |c(s,y)|\lfloor m(s)\rfloor_{\alpha}|x-y|^\alpha\mathrm ds
        \\&\leq
        \int_0^t \lVert m(s)\rVert_{C^0(\Omega)} \lfloor c(s) \rfloor_\alpha |x - y|^\alpha \mathrm ds + \int_0^t \lfloor m(s)\rfloor_{\alpha}|x-y|^\alpha\mathrm ds.
    \end{align*}
    Using again the abbreviation $\tilde m = m\theta$ and noting that $\tilde m$ has the same regularity as $m$, we can estimate the term $(C)$ in analogy to term $(B)$ by
    \begin{equation*}
        \int_0^t |m(s,y)\theta(y)c(s,y) - m(s,x)\theta(x)c(s,x)| \mathrm ds
        \leq
        \int_0^t \lVert \tilde m(s) \rVert_{C^0(\Omega)}\lfloor c(s) \rfloor_\alpha |x-y|^\alpha + \int_0^t \lfloor \tilde m(s) \rfloor_\alpha |x-y|^\alpha \mathrm ds.
    \end{equation*}
    To estimate $(D)$ we need to split the term
    \begin{equation*}
        (D) = \underbrace{\int_0^t \tilde |m(s,y)|\,\left| c(s,y)^2 - c(s,x)^2 \right|\mathrm ds}_{\eqqcolon D_1} + \underbrace{\int_0^t |c(s,x)|^2\left| \tilde m(s,y) - \tilde m(s,x) \right| \mathrm ds}_{\eqqcolon D_2}.
    \end{equation*}
    Using $c(s,y)^2 - c(s,x)^2 = c(s,y)(c(s,y)-c(s,x)) + c(s,x)(c(s,y)-c(s,x))$ and $c(s,x)\in[0,1]$, we estimate $(D_1)$
    \begin{align*}
        (D_1) &\leq \int_0^t |\tilde m(s,y)|\, |c(s,y)|\,\lfloor c(s)\rfloor_\alpha |x-y|^\alpha + |\tilde m(s,y)|\,|c(s,x)|\,\lfloor c(s) \rfloor_\alpha|x-y|^\alpha\mathrm ds
        \\&\leq
        \int_0^t 2|\tilde m(s,y)|\,\lfloor c(s) \rfloor_\alpha |x-y|^\alpha \mathrm ds
        \\&\leq
        \int_0^t 2\lVert \tilde m(s)\rVert_{C^0(\Omega)}\,\lfloor c(s) \rfloor_\alpha |x-y|^\alpha \mathrm ds
    \end{align*}
    and for $(D_2)$
    \begin{equation*}
        (D_2) \leq \int_0^t |c(s,x)|^2\lfloor \tilde m(s) \rfloor_\alpha |x-y|^\alpha \mathrm ds 
        \leq
        \int_0^t \lfloor \tilde m(s) \rfloor_\alpha |x-y|^\alpha \mathrm ds.
    \end{equation*}
    Collecting all estimates yields the claim and the proof is complete.
\end{proof}

\begin{lemma}\label{lemma:bound_for_c_k}
    Let $a_1$ and $a_2$ be functions in $L^4(I,C^\alpha(\Omega))$ with $a_1, a_2 \geq 0$. Assume that $\rho \in C^\alpha(\Omega)$ satisfies $0 < \rho < 1$ and $k_6$ and $k_7$ are positive constants. Then there exists a unique solution $c \in W^{1,2}(I,C^\alpha(\Omega))$ to the equation
    \begin{equation*}
        d_tc = k_6a_1a_2(1+k_7c)\left( 1 - \frac{c}{1-\rho} \right), \quad c(0) = 0
    \end{equation*}
    with $0\leq c \leq 1$. Furthermore, we can control the full $\alpha$-H\"older norm of $c$ in the following way
    \begin{equation*}
        \lVert c \rVert_{W^{1,2}(I,C^\alpha(\Omega))} \leq C\left( \lVert a_1 \rVert_{L^2(I,C^\alpha(\Omega))}, \lVert a_2 \rVert_{L^2(I,C^\alpha(\Omega))}, \lVert \rho \rVert_{C^\alpha(\Omega)}  \right),
    \end{equation*}
    with the constant $C$ being monotone in its arguments.
\end{lemma}
\begin{proof}
    We use again the notation
    \begin{equation*}
        d_tc = m(1+c)(1-\theta c), \quad c(0)=0,
    \end{equation*}
    where $m\in L^2(I,C^\alpha(\Omega))$ and $m(t,x) \geq 0$ and $\theta \in C^\alpha(\Omega)$. Thus, the inducing function $F$ in the sense of Theorem B.2 in \cite{dondl2021efficient} is given by
    \begin{equation*}
        F:I\times C^\alpha(\Omega) \to C^\alpha(\Omega), \quad F(t,c) = m(t)(1 + c)(1 - \theta c).
    \end{equation*}
    To prove the existence of a unique short-time solution in the space $W^{1,2}(I_\delta, C^\alpha(\Omega))$, we need $F$ to be of Carath\'eodory regularity. Clearly, $F(\cdot,c):I\to C^\alpha(\Omega)$ is Bochner measurable as $m$ is. Furthermore, $F(t,\cdot):C^\alpha(\Omega) \to C^\alpha(\Omega)$ is continuous. This is due to the fact that $C^\alpha(\Omega)$ is a Banach algebra.
    
    To proceed, we need a boundedness and a Lipschitz condition on bounded subsets of $C^\alpha(\Omega)$, compare to Theorem B.2 in \cite{dondl2021efficient}. To this end, let $B\subset C^\alpha(\Omega)$ be a bounded set. For $c \in B$ we estimate
    \begin{equation*}
        \lVert F(t,c) \rVert_{C^\alpha(\Omega)} \leq C \lVert m(t) \rVert_{C^\alpha(\Omega)}\lVert 1 + c \rVert_{C^\alpha(\Omega)}\lVert 1-\theta c \rVert_{C^\alpha(\Omega)}
    \end{equation*}
    The term $\lVert 1 + c \rVert_{C^\alpha(\Omega)}\lVert 1-\theta c \rVert_{C^\alpha(\Omega)}$ can be bounded in terms of the measure of $\Omega$, the assumed boundedness of $B$ and the $C^\alpha(\Omega)$ norm of $\theta$. Hence, there exists a constant $C = C(\Omega, \lVert \theta \rVert_{C^\alpha(\Omega)}, B)$ such that
    \begin{equation*}
        \lVert F(t,c) \rVert_{C^\alpha(\Omega)} \leq C\left( \Omega, \lVert \theta \rVert_{C^\alpha(\Omega)},B \right) \lVert m(t) \rVert_{C^\alpha(\Omega)}
    \end{equation*}
    and by assumption, the map $t\mapsto \lVert m(t) \rVert_{C^\alpha(\Omega)}$ is a member of $L^2(I)$. Now, let $c$ and $\bar c \in B$ and look at the differences
    \begin{align*}
        \lVert F(t,c) - F(t,\bar c) \rVert_{C^\alpha(\Omega)} &\leq C \lVert m(t) \rVert_{C^\alpha(\Omega)} \lVert (1+c)(1-\theta c) - (1+\bar c)(1 - \theta\bar c) \rVert_{C^\alpha(\Omega)}
        \\&\leq
        C\lVert m(t) \rVert_{C^\alpha(\Omega)}\left[ \lVert 1+ \theta \rVert_{C^\alpha(\Omega)} \lVert c - \bar c \rVert_{C^\alpha(\Omega)} + \lVert \theta \rVert_{C^\alpha(\Omega)}\left\lVert c^2 - \bar c^2 \right\rVert_{C^\alpha(\Omega)} \right].
    \end{align*}
    We look at the quadratic term separately
    \begin{align*}
        \left\lVert c^2 - \bar c^2 \right\rVert_{C^\alpha(\Omega)} \leq C \lVert c - \bar c \rVert_{C^\alpha(\Omega)}\lVert c + \bar c \rVert_{C^\alpha(\Omega)} \leq C\left(B\right)\lVert c-\bar c \rVert_{C^\alpha(\Omega)}.
    \end{align*}
    Hence, there exists a function $L_B \in L^2(I)$ such that
    \begin{equation*}
        \lVert F(t,c) - F(t,\bar c) \rVert_{C^\alpha(\Omega)} \leq L_B(t)\lVert c - \bar c \rVert_{C^\alpha(\Omega)}.
    \end{equation*}
    Consulting Theorem B.2 in \cite{dondl2021efficient}, the estimates above provide the existence of an interval $[0,\delta] = I_\delta$ and a unique function $c\in W^{1,2}(I_\delta, C^\alpha(\Omega))$ solving the ODE.
    
    To show that the solution can be extended to all of $I=[0,T]$, we extend the solution $c$ to the maximal interval $[0,t^*)$ of existence. For any $t_0\in [0,t^*)$ we set $c_0 = c(t_0)$ and consider the initial value problem
    \begin{equation*}
        d_tc = k_6a_1a_2(1+k_7c)\left(1 - \frac{c}{1-\rho} \right), c(t_0) = c_0.
    \end{equation*}
    Then this has a unique solution in $W^{1,2}([t_0-\tilde \delta, t_0 + \tilde\delta]\cap I, C^\alpha(\Omega))$ for some suitable $\tilde \delta$. In fact, $\tilde \delta$ depends on the $L^2(I,C^\alpha(\Omega))$ norm of $a_1a_2$, the $C^\alpha(\Omega)$ norm of $(1-\rho)^{-1}$ and the $L^2(I,C^\alpha(\Omega))$ norm of $c$ on $[0,t^*)$. This implies that $\tilde\delta$ does not depend on the position of $t_0\in[0,t^*)$ and thus $t^*=T$ and the interval $[0,t^*)$ can be closed.
    
    Finally, the promised bound on the $W^{1,2}(I,C^\alpha(\Omega))$ norm of $c$ is easily established using $c(t,x)\in[0,1]$ and the estimate on the H\"older seminorm of Lemma \ref{lemma:holder_bound_c_k}.
\end{proof}

We show now how to establish the existence of solutions to the bone ODE in the space $W^{1,2}(I,C^\alpha(\Omega))$. Furthermore, we show that the $W^{1,2}(I,C^\alpha(\Omega))$ norm of such solutions can be bounded, given bounded data in the right spaces. This can in principle be done by the same arguments as for the cell equation, however, the bone ODE is linear and thus we can use more elegant approaches.
\begin{lemma}\label{lemma:bound_for_b_k}
    Let $X$ be a Banach algebra and denote by $C_X > 0$ the norm of its multiplication and assume that $p > 1$. By $W^{1,p}_0(I,X)$ we denote the vector-valued Sobolev space with vanishing initial conditions. For a function $m\in L^p(I,X)$ we define the multiplication operator
    \begin{equation*}
        M: C^0(I,X) \to L^p(I,X), \quad Mv = t\mapsto m(t)v(t).
    \end{equation*}
    Then the map
    \begin{equation*}
        d_t + M: W^{1,p}_0(I,X) \to L^p(I,X), \quad v\mapsto d_tv + Mv
    \end{equation*}
    is a linear homeomorphism. Furthermore, given a right-hand side $f\in L^p(I,X)$ we may bound the solution $v$ to $d_tv + Mv = f$ in the following way
    \begin{equation*}
        \lVert v \rVert_{W^{1,p}_0(I,X)} \leq C \left( |I|, \lVert m \rVert_{L^p(I,X)}, C_X \right)\lVert f \rVert_{L^p(I,X)},
    \end{equation*}
    i.e., the norm of $v$ does only depend on $f$ and $m$ measured in $L^p(I,X)$ norm and the constant $C$ is monotone in these quantities.
\end{lemma}
\begin{proof}
    The continuity and linearity of the map $d_t + M$ is clear. Its bijectivity follows as an application of Theorem B.2 in \cite{dondl2021efficient}. To this end, note that the inducing function $F:I\times X \to X$ of Theorem B.2 in \cite{dondl2021efficient} is given by
    \begin{equation*}
        F:I\times X \to X, \quad F(t,x) = m(t)x.
    \end{equation*}
    This is clearly a Carath\'eodory function and it holds for $x,y \in X$
    \begin{equation*}
        \lVert F(t,x) - F(t,y) \rVert_X \leq C \lVert m(t) \rVert_X\lVert x-y \rVert_X.
    \end{equation*}
    The function $C\lVert m(\cdot) \rVert_X$ is a member of $L^p(I)$ with $p > 1$ and therefore the existence of a unique solution $v \in W^{1,p}_0(I,X)$ is established. To provide the bound, we employ Gr\"onwall's inequality. Note that, by the fundamental theorem, the solution $v$ satisfies the integral identity
    \begin{equation*}
        v(t) = \int_0^t f(s) - m(s)v(s)\mathrm ds
    \end{equation*}
    and consequently the estimate
    \begin{equation*}
        \lVert v(t) \rVert_X \leq \int_0^t \lVert f(s) \rVert_X + C_X\lVert m(s) \rVert_X\lVert v(s) \rVert_X\mathrm ds.
    \end{equation*}
    Using Gr\"onwall's inequality yields
    \begin{equation*}
        \lVert v(t) \rVert_X \leq \left[ 1 + C_X\lVert m \rVert_{L^1(I,X)}\exp\left( C_X \lVert m \rVert_{L^1(I,X)} \right) \right]\cdot \lVert f \rVert_{L^1(I,X)}.
    \end{equation*}
    Clearly, this implies a bound in $C^0(I,X)$ norm for $v$ of the form
    \begin{equation*}
        \lVert v \rVert_{C^0(I,X)}\leq C \left( \lVert m \rVert_{L^1(I,X)}, C_X \right)\lVert f \rVert_{L^1(I,X)}
    \end{equation*}
    and consequently also in $L^p(I,X)$. To bound $d_tv$, we use the equation satisfied by $v$ and estimate
    \begin{align*}
        \lVert d_tv \rVert_{L^p(I,X)} &= \lVert f - Mv \rVert_{L^p(I,X)}
        \\
        &\leq 
        \lVert f \rVert_{L^p(I,X)} + \lVert v \rVert_{C^0(I,X)}\lVert m \rVert_{L^p(I,X)}
        \\
        &\leq 
        C \left( |I|, \lVert m \rVert_{L^p(I,X)}, C_X \right)\lVert f \rVert_{L^p(I,X)}.
    \end{align*}
\end{proof}
\begin{lemma}\label{lemma:providing_assumptions}
    Assume $(\rho_k)$ is a minimizing sequence for $\hat J + \eta\lVert\cdot\rVert^2_{H^2(\Omega)}$. Then properties $(A1)$-$(A5)$ hold.
\end{lemma}
\begin{proof}
    We begin with $(A1)$. The regularizing term $\eta\lVert\cdot\rVert^2_{H^2(\Omega)}$ leads to a $H^2(\Omega)$ bound for any minimizing sequence $(\rho_k)$ of $\hat J + \eta\lVert\cdot\rVert^2_{H^2(\Omega)}$, as $\hat J$ is bounded from below. Then there exists a subsequence (not relabeled) with 
    \begin{equation*}
        \rho_k \rightharpoonup \rho^*\quad\text{in }H^2(\Omega).
    \end{equation*}
    Using the compactness of the embedding $H^2(\Omega)\hookrightarrow\hookrightarrow C^0(\Omega)$, we get
    \begin{equation*}
        \rho_k \to \rho^* \quad \text{in }C^0(\Omega)
    \end{equation*}
    and inferring the closedness of $P$ in $C^0(\Omega)$ yields $\rho^* \in P$ as desired.
    
    We provide first a weaker statement than $(A2)$. Namely, we prove that $(u_k)$ is bounded uniformly in $C^0(I,H^1(\Omega))$. At the end of the proof, we can show the full validity of $(A2)$. In Setting \ref{section:setting_optimal_control} we assumed that the boundary conditions satisfied by $u_k$ are
    \begin{gather*}
        \mathbb{C}(\rho_k,\sigma,b_k)\varepsilon(u_k)\cdot \eta = g_N, \quad u_k=u_D
    \end{gather*}
    on $\Gamma_N$ and $\Gamma_D$ respectively, where $g_N \in C^0(I,L^2(\partial\Omega)) \subset C^0(I,H^{1/2}(\partial\Omega)^*)$ and $u_D\in C^0(I,H^{1+\theta}(\Omega))$ implying that $(u_D)_{|\Gamma_D}\in H^{1/2}(\partial\Omega)$. The unique solutions $\tilde u_k \in L^2(I,H^1_D(\Omega))$ to 
    \begin{equation}\label{equation:equation_for_u_0_k}
        \iint \mathbb{C}(\rho_k, \sigma, b_k)\varepsilon(\tilde u_k):\varepsilon(\cdot)\mathrm dx\mathrm dt = \underbrace{\int_I\langle g_N, \cdot \rangle_{H^{1/2}(\partial\Omega)} \mathrm dt - \iint \mathbb{C}(\rho_k, \sigma, b_k)\varepsilon(u_D):\varepsilon(\cdot)\mathrm dx\mathrm dt}_{\eqqcolon f_k \in L^2(I,H^1_D(\Omega))^*} \quad \text{in }L^2(I,H^1_D(\Omega))^*
    \end{equation}
    have therefore right-hand sides $f_k$ that can be interpreted as members of $C^0(I,H^1_D(\Omega)^*)$. This is due to the assumption $g_N \in C^0(I,H^{1/2}(\partial\Omega)^*)$ and a standard computation that shows that the map 
    \begin{equation*}
        t \mapsto \int_\Omega \mathbb{C}(\rho_k,\sigma, b_k)(t)\varepsilon(u_D(t)):\varepsilon(\cdot)\mathrm dx
    \end{equation*}
    is a member of $C^0(I,H^1_D(\Omega)^*)$. Hence Lemma \ref{lemma:bound_for_u_k} is applicable and shows that
    \begin{equation*}
        \lVert \tilde u_k \rVert_{C^0(I,H^1(\Omega))} \leq C\left( \lfloor \mathbb{C}(\rho_k,\sigma,b_k) \rfloor, C_{\text{Korn}} \right) \lVert f_k \rVert_{C^0(I,H^1_D(\Omega)^*)}.
    \end{equation*}
    As $\mathbb{C}$ is coercive and essentially bounded, this estimate can be made independent of $k\in \mathbb{N}$. 
    Clearly, we have not yet proven $(A2)$ but will first continue with the other assumptions.
    
    We are concerned with $(A3)$ now which is an application of Lemma \ref{lemma:bound_for_a_k}. The corresponding Gelfand triple is $(\operatorname{Id},H^1_D(\Omega),L^2(\Omega))$ and the operators $M_k$ are
    \begin{equation*}
        M_k:H^1_D(\Omega) \to H^1_D(\Omega)^*, \quad M_ka = \int_\Omega D(\rho_k)\nabla a\nabla\cdot + k_3a\cdot\mathrm dx.
    \end{equation*}
    The coercivity constant of $M_k$ can be estimated from below independently of $(\rho_k)$ by 
    \begin{equation*}
        \lfloor M_k \rfloor \geq \min(\lfloor D(\rho_k) \rfloor, k_3).
    \end{equation*}
    On the other hand, the operator norm of $M_k$ can be estimated to
    \begin{equation*}
        \lVert M_k \rVert = \sup_{\lVert a \rVert \leq 1, \lVert \varphi \rVert \leq 1}\int_\Omega D(\rho_k)\nabla a \nabla \varphi + k_3 a\varphi\mathrm dx \leq \lVert D(\rho_k) \rVert_{L^\infty(\Omega,\mathcal{M}_s)} + k_3.
    \end{equation*}
    The right-hand sides of the equation are given by 
    \begin{equation*}
        f_k = \iint \left( k_2S(\varepsilon(u_k)) c_k - k_3 \right) \cdot \mathrm dx \mathrm dt,
    \end{equation*}
    consequently their norm can be estimated
    \begin{align*}
        \lVert f_k \rVert_{L^2(I,H^1_D(\Omega)^*)} 
        & \leq 
        \lVert f_k \rVert_{L^2(I,L^2(\Omega)^*)}
        \\ &\leq 
        \left[ \iint \left( k_2S(\varepsilon(u_k)) c_k - k_3 \right) \right]^{1/2}
        \\ &\leq
        k_2\lVert c_k \rVert_{C^0(I\times\Omega)}\lVert u_k \rVert_{L^2(I,H^1_D(\Omega))} + |I\times\Omega|^{1/2}k_3.
    \end{align*}
    This is uniformly bound in $k\in \mathbb{N}$ by the bound on $(u_k)$ and the pointwise properties of $c_k$, i.e., $0\leq c_k \leq 1$. Thus we apply Lemma \ref{lemma:bound_for_a_k} to obtain
    \begin{align*}
        \lVert a_k \rVert_{H^1(I,H^1_D(\Omega), H^1_D(\Omega)^*)} &\leq C\left( \lVert M_k \rVert, \lfloor M_k\rfloor^{-1} \right)\cdot\left( \lVert a_k(0) \rVert_{L^2(\Omega)} + \lVert f_k \rVert_{L^2(I,H^1_D(\Omega))} \right)
        \\&\leq
        C\left( \lVert D(\rho_k) \rVert_{L^\infty(\Omega,\mathcal{M}_s)}, \lfloor D(\rho_k) \rfloor^{-1}, k_3 \right)\cdot\left( \lVert a_k(0) \rVert_{L^2(\Omega)} + \lVert f_k \rVert_{L^2(I,H^1_D(\Omega))} \right)
    \end{align*}
    By the ellipticity and boundedness of $D$, the constant initial conditions $\tilde a^i_k(0)\equiv 1$ and the estimate for $u_k$, we see that $(a^i_k)$ are bounded uniformly in $k\in\mathbb{N}$. Using the reflexivity of the Hilbert space $H^1(I,H^1_D(\Omega), H^1_D(\Omega)^*)$ to produce a weakly convergent subsequence with limit $\tilde a_i^*$ we proved $(A3)$.
    
    We proceed with $(A4)$ and aim to apply Lemma \ref{lemma:bound_for_b_k} to the ODE
    \begin{equation}\label{equation:repeat_bone_ode}
        d_tb_k = k_4a^1_k\left( 1 + \frac{b_k}{1-\rho_k} \right), \quad b_k(0) = 0,
    \end{equation}
    with $X = C^\alpha(\Omega)$ and $p = 2$. To this end we rearrange \eqref{equation:repeat_bone_ode} to 
    \begin{equation*}
        d_tb_k - \frac{k_4 a^1_k}{1-\rho_k}b_k = k_4a^1_k,
    \end{equation*}
    thus
    \begin{gather*}
        m_k = \frac{k_4 a^1_k}{1 - \rho_k}\quad \text{and}\quad f_k = k_4a^1_k
    \end{gather*}
    in the notation of Lemma \ref{lemma:bound_for_b_k}. Due to the embedding $H^2(\Omega)\hookrightarrow\hookrightarrow C^\alpha(\Omega)$ in three spatial dimensions, we get $\rho_k\in C^\alpha(\Omega)$ and also $(1 - \rho_k)^{-1}\in C^\alpha(\Omega)$ with a uniform bound in H\"older norm. Then, Lemma \ref{lemma:bound_for_b_k} guarantees that $(a^1_k)$ is bounded uniformly in $L^2(I,C^\alpha(\Omega))$ which implies such a bound for $(m_k)$ and $(f_k)$. We may therefore use Lemma \ref{lemma:bound_for_b_k} to obtain
    \begin{equation*}
        \lVert b_k \rVert_{W^{1,2}(I,C^\alpha(\Omega))} \leq C\left( I, C_{C^\alpha(\Omega)}, \lVert m_k \rVert_{L^2(I,C^\alpha(\Omega))} \right)\lVert f_k \rVert_{L^2(I,C^\alpha(\Omega))}
    \end{equation*}
    and guarantee that the bound is independent of $k\in \mathbb{N}$. Furthermore, we have the compact embedding 
    \begin{equation*}
        W^{1,2}(I,C^\alpha(\Omega))\hookrightarrow\hookrightarrow C^0(I\times\Omega).
    \end{equation*}
    This yields the relative compactness of $(b_k)$ in $C^0(I\times\Omega)$ and thus the existence of $b^*\in C^0(I\times\Omega)$ and a (not relabeled) subsequence with
    \begin{equation*}
        b_k \to b^*.
    \end{equation*}
    To provide the existence of $c^*$ and a subsequence $c_k\to c^*$ in $C^0(I\times\Omega)$, we note that Lemma \ref{lemma:bound_for_c_k} provides a bound of the $W^{1,2}(I,C^\alpha(\Omega))$ norm of $c_k$ of the form
    \begin{equation*}
        \lVert c_k \rVert_{W^{1,2}(I,C^\alpha(\Omega))} \leq C \left( \lVert a_1^k \rVert_{L^2(I,C^\alpha(\Omega))}, \lVert a_2^k \rVert_{L^2(I,C^\alpha(\Omega))}, \lVert \rho \rVert_{C^\alpha(\Omega)} \right)
    \end{equation*}
    with $C$ being increasing in its arguments. As we proved suitable bounds for $(a_1^k)$, $(a_2^k)$ and $(\rho_k)$ this yields a $W^{1,2}(I, C^\alpha(\Omega))$ bound for $(c_k)$ that does not depend on $k\in\mathbb{N}$. Therefore, the sequence $(c_k)$ is relatively compact in $C^0(I\times\Omega)$ and $(A5)$ follows.
    
    We are still left with showing the existence of a subsequence of $(u_k)$ and a function $u^*\in C^0(I,H^1(\Omega))$ such that
    \begin{equation*}
        u_k \to u^*\quad\text{in }C^0(I,H^1(\Omega)).
    \end{equation*}
    To this end, note that we have established that $(b_k)$ is relatively compact in $C^0(I,C^0(\Omega))$. Hence, applying the Arzel\`a-Ascoli Theorem, $(b_k)$ is equi-continuous as well. Now, going back to \eqref{equation:equation_for_u_0_k} we can easily compute that the sequence $(f_k)$ is indeed equi-continuous in $C^0(I,H^1_D(\Omega)^*)$. Again, this is essentially due to the equi-continuity of $(b_k)$ in the space $C^0(I,C^0(\Omega))$. Hence, applying Lemma \ref{lemma:equi_continuity_of_uk} yields the equi-continuity of $(u_k)$ in $C^0(I,H^1(\Omega))$. In view of the Arzel\`a-Ascoli Theorem we still need to show that the sets
    \begin{equation*}
        \{ u_k(t) \in H^1(\Omega)\mid k\in\mathbb{N} \}
    \end{equation*}
    are relatively compact in $H^1(\Omega)$. This can be established by looking at the equation satisfied by $u_k(t)$ for every fixed $t\in I$ and applying Lemma \ref{lemma:higher_regularity_elliptic_equation}. Indeed, $u_k(t) =\tilde u_k(t) + u_D(t)$ satisfies
    \begin{equation*}
        \int_\Omega \mathbb{C}(\rho_k,\sigma,b_k)(t)\varepsilon(u_0^k(t)):\varepsilon(\cdot)\mathrm dx = \int_{\partial\Omega}g_N(t)\cdot\mathrm ds - \int_\Omega\mathbb{C}(\rho_k,\sigma,b_k)(t)\varepsilon(u_D(t)):\varepsilon(\cdot)\mathrm dx.
    \end{equation*}
    The assumptions on $g_N$ and $u_D$ guarantee that the right-hand side lies in $H^{1-\theta}(\Omega)^*$ and the H\"older regularity established for $(b_k)$, i.e., $(b_k(t))\subset C^\alpha(\Omega)$ allows to deduce the $H^{1+\theta}(\Omega)$ regularity for $u_k(t)$. Additionally, the uniform $W^{1,2}(I,C^\alpha(\Omega))$ bound for the sequence $(b_k)$ established before yields a bound for the $C^\alpha(\Omega)$ norm of $(b_k(t))$ and also the coefficients of $\mathbb{C}(\rho_k,\sigma,b_k)$ via assumption \eqref{equation:holder_regulariy_of_coefficients}. Collecting these bounds in fact implies that \begin{equation*}
        \sup_{k\in \mathbb{N}}\lVert u_k(t) \rVert_{H^{1+\theta}(\Omega)} \leq C
    \end{equation*}
    and invoking the compactness result of Rellich which states that
    \begin{equation*}
        H^{1+\theta}(\Omega)\hookrightarrow\hookrightarrow H^1(\Omega)
    \end{equation*}
    we can conclude the missing piece in order to apply the Arzel\`a-Ascoli Theorem to $(u_k)$ in the space $C^0(I,H^1_D(\Omega))$. This eventually guarantees the validity of assumption $(A2)$.
\end{proof}
We can now conclude the section by providing the proofs of Theorem \ref{theorem:optimal_control_approx_objective} and Corollary \ref{corollary:optimal_control_numerical_penalization}. 
\begin{proof}[Proof of Theorem \ref{theorem:optimal_control_approx_objective}]
    Let $(\rho_k)\subset P$ be a minimizing sequence for $\hat J + \eta\lVert\cdot\rVert^2_{H^2(\Omega)}$. Lemma \ref{lemma:providing_assumptions} shows that the assumptions $(A1)- (A5)$ hold and Proposition \ref{proposition:proof_under_assumptions} shows that this leads to the existence of an accumulation point $\rho^*\in P$ of the sequence $(\rho_k)$ which is a minimizer of $\hat J + \eta\lVert \cdot \rVert_{H^2(\Omega)}^2$.
\end{proof}
\begin{proof}[Proof of Corollary \ref{corollary:optimal_control_numerical_penalization}]
    Let $(\rho_k)\subset H^2(\Omega)$ be a minimizing sequence for $\hat J + \eta\lVert\cdot\rVert^2_{H^2(\Omega)} + \mathcal{K}$. Revisiting the proof of Proposition \ref{proposition:proof_under_assumptions} shows that the additional term $\mathcal{K}$ does not lead to complications in the lower semicontinuity as it is assumed to be continuous on $C^0(I\times\Omega)$, i.e., a compact perturbation. Furthermore, as we assumed that $\mathcal{K}$ takes non-negative values only, also the coercivity of the objective function is not violated through the addition of $\mathcal{K}$.
\end{proof}

\end{document}